\documentclass[a4paper,10pt]{article}
\usepackage[utf8]{inputenc}
\usepackage{mathrsfs,pgf,tikz,amsmath,amssymb,hyperref,amsthm,stmaryrd}
\usepackage{colortbl}
\usetikzlibrary{arrows}

\theoremstyle{plain}
\newtheorem{theorem}{Theorem}[section]                                          
\newtheorem{proposition}[theorem]{Proposition}                          
\newtheorem{lemma}[theorem]{Lemma}

\theoremstyle{definition}

\theoremstyle{remark}
\newtheorem{remark}[theorem]{Remark}

\newtheorem{examples}[theorem]{Examples}

\DeclareMathOperator{\p}{P}

\title{Benford or not Benford: a systematic but not always well-founded use of an elegant law in experimental fields}
\author{Blondeau Da Silva Stéphane}

\begin{document}

\maketitle

\begin{abstract}
In this paper, we will see that the proportion of $d$ as leading digit, $d\in\llbracket1,9\rrbracket$, in data (obtained thanks to the hereunder 
developed model) is more likely to follow a law whose probability distribution is determined by a specific upper bound, rather than Benford's Law. These 
probability distributions fluctuate around Benford's value as can often be observed in the literature in many naturally occurring collections of data 
(where the physical, biological or economical quantities considered are upper bounded). Knowing beforehand the value of the upper bound can be a way to 
find a better adjusted law than Benford's one. 
\end{abstract}



\section*{Introduction}

Benford's Law, also called Newcomb-Benford's Law, is noteworthy to say the least: according to it, the first digit $d$, $d\in\llbracket1,9\rrbracket$, of 
numbers in many naturally occurring collections of data does not follow a discrete uniform distribution, as might be thought, but a logarithmic 
distribution. Discovered by the astronomer Newcomb in $1881$ (\cite{new}), this law was definitively brought to light by the physicist Benford in 
$1938$ (\cite{ben}). He proposed the following probability distribution: the probability for $d$ to be the first digit of a number seems to be equal to 
$\log(d+1)-\log(d)$, \textit{i.e.} $\log(1+\frac{1}{d})$. Benford tested it over data set from $20$ different domains (surface areas of rivers, sizes of
american populations, physical constants, molecular weights, entries from a mathematical handbook, numbers contained in an issue of Reader's Digest, the 
street addresses of the first persons listed in American Men of Science, death rates, \textit{etc.}). Most of the empirical data, as physical data 
(Knuth in \cite{knu} or Burke and Kincanon in \cite{BK}), economic and demographic data (Nigrini and Wood in \cite{NW}) or genome data (Friar 
\textit{et al.} in \cite{FGP}), follow approximately Benford's Law. To such an extent that this law is used to detect possible frauds in lists of 
socio-economic data (\cite{var}) or in scientific publications (\cite{die}).

First restricted to the experimental field, it is now established that this law holds for various mathematical sequences (see for example \cite{ber}). 
In the situation, where the distribution of first digits is scale, unit or base invariant, this distribution is always given by Benford's Law 
(\cite{pin} and \cite{hil}). Selecting different samples in different populations, under certain constraints, leads also to construct a sequence that 
follows the Benford's Law (\cite{hip}). Furthermore independant variables multiplication conducts to this law (\cite{boy}). One might add that some 
sequences satisfy Benford's Law exactly (for example see \cite{sar},\cite{was} or \cite{jol}).

We can note that there also exist distributions known to disobey Benford's Law (\cite{rai} and \cite{bee}). And even concerning empirical data sets, 
this law appears to be a good approximation of the reality, but no more than an approximation (\cite{del}).

In the model we build in the article, the naturally occurring data will be considered as the realizations of independant random variables following the 
hereinafter constraints: $(a)$ the data is strictly positive and is upper-bounded by an integer $n$, constraint which is often valid in data sets, the 
physical, biological and economical quantities being limited ; $(b)$ each random variable is considered to follow a discrete uniform distribution 
whereby the first $i$ strictly positive integers are equally likely to occur ($i$ being uniformly randomly selected in $\llbracket1,n\rrbracket$). This 
model relies on the fact that the random variables are not always the same. The article is divided into two parts. In the first one, we will accurately 
study the case where the leading digit is $1$. In the second one, we will generalize our results to the eight last cases.

Through this article we will demonstrate that the predominance of $1$ as first digit (followed by those of $2$ and so on) is all but surprising, and 
that the observed fluctuations around the values of probability determined by Benford's Law are also predictible. The point is that, since $1938$, 
Benford's Law probabilities became standard values that should exactly be followed by most of naturally occurring collections of data. However the 
reality is that the proportion of each $d$ as leading digit, $d\in\llbracket1,9\rrbracket$, structurally fluctuates. There is not a single Benford's Law 
but numerous distinct laws that we will hereafter examine.

\section{The chosen probability space}

\subsection{Notations}

In order to determine the proportion of numbers whose leading digit is $d\in\llbracket1,9\rrbracket$, we will first build our probability space and 
further explain the model we choose.

Let $i$ be a strictly positive integer. Let $\mathcal{U}_{\{i\}}$ denote the discrete uniform distribution whereby the first $i$ strictly positive 
integers are equally likely to be observed.

Let $n$ be a strictly positive integer. Let us consider the random experiment $\mathcal{E}_n$ of tossing two independent dice. The first 
one is a fair $n$-sided die showing $n$ different numbers from $1$ to $n$. The number $i$ rolled on it defines the number of faces on the second die. 
It thus shows $i$ different numbers from $1$ to $i$. 

Let us now define the probability space $\Omega_n$ as follows: $\Omega_n=\{(i,j):i\in\llbracket1,n\rrbracket\text{ and }j\in\llbracket1,i\rrbracket\}$. 
Our probability measure is denoted by $\p$. 

Let us denote by $L_{n}$ the random variable from $\Omega_n$ to $\llbracket1,9\rrbracket$ that maps each element $\omega$ of $\Omega_n$ to the leading 
digit of the second component of $\omega$.

\subsection{Why such a model?}

Let us imagine a perfect consumer shopping in a perfectly structured supermarket: (a) in the $i^{\text{th}}$ ($i$ being a strictly positive integer) 
section of this supermarket, the products prices range between $1$ and $i$ cents of the considered currency; (b) the prices in a section are uniformly 
distributed; (c) each section contains the same quantity of products and (d) the consumer randomly chooses his products in the whole store. 

Under these constraining hypotheses, these perfect entities enable us to use our model. Note that conditions (c) and (d) gathered avoid us to conduct a 
double drawing every time: first the section then the product. In that respect, the sales receipt will verify the following results in terms of 
proportion of $d$ as leading digit, $d\in\llbracket1,9\rrbracket$.

Among the different domains studied by Benford (\cite{ben}), some could be well adapted to our model: sizes of populations (sections here
gathering all the populations having the same usable areas, the geographic constraints preventing the surface area to be broader; populations being not 
neccessary settled on the entire area, their sizes fluctuate) or street adresses for example (sections here gathering the adresses of a selected street; 
the lenght of the considered streets might be uniformly distributed to fit model criteria).

Hence the defined model is relevant when the studied data can be considered as realizations of a homogeneous and expanded range of random variables 
approximately following discrete uniform distributions.

\section{Proportion of \texorpdfstring{$d$}{}}

Through the below proposition, we will express the probability $\p(L_{n}=d)$, for each $n\in\mathbb N^*$, \textit{i.e.} the probability that the 
leading digit of our second throw in our random experiment is $d$.

\begin{proposition}\label{proi}Let $k$ denote the positive integer such that $k=\min\{i\in\mathbb N:d\times10^i>n\}$.
If $n< (d+1)\times10^{k-1}$, the value of $\p(L_{n}=d)$ is:
\begingroup\scriptsize
\begin{equation*}
\frac{1}{n}\big(\sum_{l=0}^{k-2}(\sum_{b=d\times10^l}^{(d+1)\times10^l-1}\frac{b-\frac{(9d-1)\times10^l-8}{9}}{b}+\sum_{a=(d+1)\times10^l}^{d
\times10^{l+1}-1}\frac{\frac{10^{l+1}-1}{9}}{a})+\sum_{b=d\times10^{k-1}}^{n}\frac{b-\frac{(9d-1)\times10^{k-1}-8}{9}}{b}\big).
\end{equation*}
\endgroup
Otherwise the value of $\p(L_{n}=d)$ is:
\begingroup\scriptsize
\begin{align*}
&\frac{1}{n}\Big(\sum_{l=0}^{k-2}(\sum_{b=d\times10^l}^{(d+1)\times10^l-1}\frac{b-\frac{(9d-1)\times10^l-8}{9}}{b}+\sum_{a=(d+1)\times10^l}^{d\times
10^{l+1}-1}\frac{\frac{10^{l+1}-1}{9}}{a})\\
&\enspace+\sum_{b=d\times10^{k-1}}^{(d+1)\times10^{k-1}-1}\frac{b-\frac{(9d-1)\times10^{k-1}-1}{9}}{b}+\sum_{a=(d+1)\times10^{k-1}}^{n}
\frac{\frac{10^k-1}{9}}{a}\Big)\,.
\end{align*}
\endgroup
\end{proposition}

\begin{proof}Let us denote by $D_{n}$ the random variable from $\Omega_n$ to $\llbracket1,n\rrbracket$ that maps each element $\omega$ of $\Omega_n$ to 
the first component of $\omega$. It returns the number obtained on the first throw of the unbiased $n$-sided die. For each $i\in\llbracket1,n
\rrbracket$, we have:
\begin{align}\label{equi}
\p(D_n=i)=\frac{1}{n}\,.
\end{align}

According to the Law of total probability, we state:
\begin{align}\label{espi}
\p(L_{n}=d)=\sum_{i=1}^n\p(L_n=d|D_n=i)\p(D_n=i)\,.
\end{align}

Thereupon two cases appear in determining the value of $\p(L_n=d|D_n=i)$, for $i\in\llbracket1,n\rrbracket$. Let us study the first case where the 
leading digit of $i$ is $d$. Let $k_i$ be the positive integer such that $k_i=\min\{k\in\mathbb N:d\times10^k> i\}$ in both cases. Among the first 
$d\times10^{k_i-1}-1$ non-zero integers (all lower than $i$), the number of integers whose leading digit is $d$ is (if $k_i\ge2$):
\begin{equation*}
\sum_{t=0}^{k_i-2}10^t=1\times\frac{1-10^{k_i-1}}{1-10}=\frac{10^{k_i-1}-1}{9}\,.
\end{equation*}
This equality still holds true for $k_i=1$. From $d\times10^{k_i-1}$ to $i$, there exist $i-d\times10^{k_i-1}+1$ additional integers whose leading digit 
is $d$. It may be inferred that:
\begin{equation}\label{p1i}
\p(L_n=d|D_n=i)=\frac{1}{i}(\frac{10^{k_i-1}-1}{9}+i-d\times10^{k_i-1}+1)=\frac{i-\frac{(9d-1)10^{k_i-1}-8}{9}}{i},
\end{equation}
the leading digit of $i$ being here $d$.

In the second case, we consider the integers $i$ whose leading digits are different from $d$. Among the first $(d+1)\times10^{k_i-1}-1$ non-zero 
integers ($i$ is greater than or equal to $(d+1)\times10^{k_i-1}$), the number of integers whose leading digit is $d$ is:
\begin{equation*}
\sum_{t=0}^{k_i-1}10^t=\frac{10^{k_i}-1}{9}\,.
\end{equation*}
From $2\times10^{k_i-1}$ to $i$, there exists no additional integers whose leading digit is $d$. It can be concluded that:
\begin{equation}\label{p2i}
\p(L_n=d|D_n=i)=\frac{\frac{10^{k_i}-1}{9}}{i},
\end{equation}
the leading digit of $i$ being here different from $d$.\\
Using equalities (\ref{equi}), (\ref{espi}), (\ref{p1i}) and (\ref{p2i}), we get our result.
\end{proof}

For example, we get:
\begin{examples}
If $n=20$, we have $k=2$. The value of $\p(L_{20}=1)$ is then (second case of Proposition \ref{proi}):
\begin{align*}
\p(L_{20}=1)&=\frac{1}{20}\Big(\frac{1-8\frac{10^0-1}{9}}{1}+\sum_{a=2}^{9}\frac{\frac{10^{0+1}-1}{9}}{a}+\sum_{b=10}^{19}\frac{b-8\frac{10^{1}-1}{9}}{b}+
\frac{\frac{10^2-1}{9}}{20}\Big)\\
&=\frac{1}{20}\Big(1+\frac{1}{2}+...+\frac{1}{9}+\frac{2}{10}+...+\frac{11}{19}+\frac{11}{20}\Big)\\
&\approx0.381\,.
\end{align*}
The value of $\p(L_{86}=8)$ is (first case of Proposition \ref{proi}):
\begingroup\scriptsize
\begin{align*}
\p(L_{86}=8)&=\frac{1}{806}\Big(\sum_{l=0}^{1}(\sum_{b=8\times10^l}^{9\times10^l-1}\frac{b-\frac{(9\times8-1)10^l-8}{9}}{b}+\sum_{a=9\times10^l}^{8\times
10^{l+1}-1}\frac{\frac{10^{l+1}-1}{9}}{a})+\sum_{b=800}^{806}\frac{b-\frac{71\times10^2-8}{9}}{b}\Big)\\
&=\frac{1}{86}\Big(\frac{1}{8}+\frac{1}{9}+...+\frac{1}{79}+\frac{2}{80}+...+\frac{11}{89}+\frac{11}{90}+...+\frac{11}{799}+\frac{12}{800}+...+
\frac{18}{806}\Big)\\
&\approx0,034\,.
\end{align*}
\endgroup
\end{examples}

\section{Study of two subsequences}

It is natural that we take a specific look at the values of $n$ positioned just before a long sequence of numbers whose leading digit is $d$ ; or 
conversely, at those positioned just before a long sequence of numbers whose leading digit is all but $d$.

To this end we will consider the sequence $(\p(L_{n}=d))_{n\in\mathbb N^*}$. In the interests of simplifying notation, we will denote by $(P_{(d,n)})_
{n\in\mathbb N^*}$ this sequence. Let us study two of its subsequences. 

\subsection{The first subsequence}

The first one is the subsequence $(P_{(d,\phi_d(n))})_{n\in\mathbb N^*}$ where 
$\phi_d$ is the function from $\mathbb N^*$ to $\mathbb N$ that maps $n$ to $d\times10^n-1$. We get the below result:
\begin{proposition}\label{sub1i}
The subsequence $(P_{(d,\phi_d(n))})_{n\in\mathbb N^*}$ converges to:
\begin{align*}
\frac{9-(9d-1)\ln(\frac{d+1}{d})+10\ln(\frac{10d}{d+1})}{81d}=\frac{9+10\ln10+9(d+1)\ln(1-\frac{1}{d+1})}{81d}\,.
\end{align*}
\end{proposition}

\begin{proof}
Let $n$ be a positive integer such that $n\ge2$. According to Proposition \ref{proi}, we have:
\begingroup\small
\begin{align*}
P_{(d,\phi_d(n))}&=P_{(d,d\times10^n-1)}\\
&=\frac{1}{d\times10^n-1}\sum_{l=0}^{n-1}(\sum_{b=d\times10^l}^{(d+1)\times10^l-1}\frac{b-\frac{(9d-1)\times
10^l-8}{9}}{b}+\sum_{a=(d+1)\times10^l}^{d\times10^{l+1}-1}\frac{\frac{10^{l+1}-1}{9}}{a})\,.
\end{align*}
\endgroup
Let us first find an appropriate lower bound of $P_{(d,\phi_d(n))}$:
\begingroup\scriptsize
\begin{align*}
P_{(d,\phi_d(n))}&\ge\frac{1}{d\times10^n}\sum_{l=1}^{n-1}(\sum_{b=d\times10^l}^{(d+1)\times10^l-1}1-\frac{(9d-1)10^l}{9}\sum_{b=d\times10^l}^{
(d+1)\times10^l-1}\frac{1}{b}+\frac{10^{l+1}-1}{9}\sum_{a=(d+1)\times10^l}^{d\times10^{l+1}-1}\frac{1}{a})\\
&\ge\frac{1}{d\times10^n}\sum_{l=1}^{n-1}\Big(10^l-\frac{9d-1}{9}10^l\ln(\frac{(d+1)\times10^l-1}{d\times10^l-1})+\frac{10^{l+1}-1}{9}
\ln(\frac{d\times10^{l+1}}{(d+1)\times10^l})\Big),
\end{align*}
\endgroup
knowing that for all integers $(p,q)$, such that $1<p<q$: 
\begin{align}\label{inelog}
\ln(\frac{q+1}{p})\le\sum_{k=p}^q\frac{1}{k}\le\ln(\frac{q}{p-1})\,.
\end{align}
Therefore we have:
\begingroup\scriptsize
\begin{align*}
P_{(d,\phi_d(n))}&\ge\frac{1}{d\times10^n}\Big(\sum_{l=1}^{n-1}10^l-\frac{9d-1}{9}\sum_{l=1}^{n-1}10^l\big(\ln(\frac{d+1}{d})+
\ln(\frac{10^l-\frac{1}{d+1}}{10^l-\frac{1}{d}})\big)\\
&\enspace+\frac{\ln(\frac{10d}{d+1})}{9}\sum_{l=1}^{n-1}(10^{l+1}-1)\Big)\\
&\ge\frac{1}{d\times10^n}\Big(\frac{10(10^{n-1}-1)}{9}-\frac{(9d-1)\ln(\frac{d+1}{d})}{9}\times\frac{10(10^{n-1}-1)}{9}\\
&\enspace-\frac{9d-1}{9}\sum_{l=1}^{n-1}10^l\ln(1+\frac{\frac{1}{d(d+1)}}{10^l-\frac{1}{d}})+\frac{\ln(\frac{10d}{d+1})}{9}\times
(\frac{10^2(10^{n-1}-1)}{9}-(n-1))\Big)\\
&\ge\frac{9-(9d-1)\ln(\frac{d+1}{d})+10\ln(\frac{10d}{d+1})}{81d}\\
&\enspace-\frac{9+10\ln(\frac{10d}{d+1})+\frac{9\ln(\frac{10d}{d+1})}{10}n}{81d\times10^{n-1}}
-\frac{9d-1}{9}\sum_{l=1}^{n-1}\frac{10^l\ln(1+\frac{\frac{1}{d+1}}{d\times10^l-1})}{d\times10^n}
\end{align*}
\endgroup
we know that for all $x\in]-1;+\infty[$ we have: $\ln(1+x)\le x$, thus:
\begingroup\scriptsize
\begin{align*}
-\frac{9d-1}{9}\sum_{l=1}^{n-1}\frac{10^l\ln(1+\frac{\frac{1}{d+1}}{d\times10^l-1})}{d\times10^n}&\ge-\frac{1}{10^n}\sum_{l=1}^{n-1}\frac{\frac{1}{d+1}
\times10^l}{d\times10^l-1}\ge-\frac{1}{10^n}\sum_{l=1}^{n-1}1\ge-\frac{n}{10^n}\,.
\end{align*}
\endgroup
Consequently, we obtain this lower bound:
\begingroup\scriptsize
\begin{align}\label{lboui}
P_{(d,\phi_d(n))}&\ge\frac{9-(9d-1)\ln(\frac{d+1}{d})+10\ln(\frac{10d}{d+1})}{81d}-\frac{90+100\ln(\frac{10d}{d+1})+9\ln(\frac{10d}{d+1})n+81dn}
{81d\times10^n}\,.
\end{align}
\endgroup
Let us now find an appropriate upper bound of $P_{(d,\phi_d(n))}$:
\begingroup\scriptsize
\begin{align*}
P_{(d,\phi_d(n))}&\le\frac{1}{d\times10^n-1}\sum_{l=0}^{n-1}(\sum_{b=d\times10^l}^{(d+1)\times10^l-1}1-\frac{(9d-1)10^l-8}{9}\sum_{b=d\times10^l}^{(d+1)
\times10^l-1}\frac{1}{b}\\
&\enspace+\frac{10^{l+1}}{9}\sum_{a=(d+1)\times10^l}^{d\times10^{l+1}-1}\frac{1}{a})\\
&\le\frac{1}{d\times10^n-1}\Big(\frac{10^n-1}{9}-\frac{1}{9}\sum_{l=1}^{n-1}\big((9d-1)10^l-8\big)\ln(\frac{(d+1)\times10^l}{d\times10^l})\\
&\enspace+\sum_{l=0}^{n-1}\frac{10^{l+1}}{9}\ln(\frac{d\times10^{l+1}-1}{(d+1)\times10^l-1})\Big),
\end{align*}
\endgroup
thanks to inequalities (\ref{inelog}). Thus we get:
\begingroup\scriptsize
\begin{align*}
P_{(d,\phi_d(n))}&\le\frac{1}{d\times10^n-1}\big(\frac{10^n}{9}-\frac{\ln(\frac{d+1}{d})}{9}\sum_{l=1}^{n-1}\big((9d-1)10^l-8\big)\\
&\enspace+\sum_{l=0}^{n-1}\frac{10^{l+1}}{9}(\ln(\frac{10d}{d+1})+\ln(\frac{10^l-\frac{1}{10d}}{10^l-\frac{1}{d+1}}))\big)\\
&\le\frac{1}{d\times10^n-1}\big(\frac{10^n}{9}-\frac{\ln(\frac{d+1}{d})}{9}\big(\frac{(9d-1)10(10^{n-1}-1)}{9}-8n\big)+
\frac{10\ln(\frac{10d}{d+1})}{9}\times\frac{10^n-1}{9}\\
\enspace&+\frac{10}{9}\sum_{l=0}^{n-1}10^l\ln(1+\frac{\frac{1}{d+1}-\frac{1}{10d}}{10^l-\frac{1}{d+1}})\big)\\
&\le\frac{1}{d(10^n-\frac{1}{d})}\big(\frac{10^n-\frac{1}{d}+\frac{1}{d}}{9}-\frac{\ln(\frac{d+1}{d})}{9}\times(\frac{(9d-1)(10^n-\frac{1}{d}
+\frac{1}{d}-10)}{9}-8n)\\
&\enspace+\frac{10\ln(\frac{10d}{d+1})}{9}\times\frac{10^n-\frac{1}{d}+\frac{1}{d}-1}{9}+\frac{10}{9}\sum_{l=0}^{n-1}10^l\ln(1+\frac{\frac{1}{d+1}-
\frac{1}{10d}}{10^l-\frac{1}{d+1}})\big)\\
&\le\frac{9-(9d-1)\ln(\frac{d+1}{d})+10\ln(\frac{10d}{d+1})}{81d}+\frac{\frac{9}{d}+90d\ln(\frac{d+1}{d})+72n\ln(\frac{d+1}{d})+
\frac{10\ln(\frac{10d}{d+1})}{d}}{81(d\times10^n-1)}\\
&\enspace+\frac{10}{9(d\times10^n-1)}\sum_{l=0}^{n-1}\frac{10^l(1-\frac{d+1}{10d})}{(d+1)10^l-1}\\
&\le\frac{9-(9d-1)\ln(\frac{d+1}{d})+10\ln(\frac{10d}{d+1})}{81d}+\frac{\frac{9}{d}+90d\ln(\frac{d+1}{d})+72n\ln(\frac{d+1}{d})+
\frac{10\ln(\frac{10d}{d+1})}{d}}{81(d\times10^n-1)}\\
&\enspace+\frac{10}{9(d\times10^n-1)}\sum_{l=0}^{n-1}1\,.
\end{align*}
\endgroup
The last step is easy to demonstrate even for $l=0$. Consequently, we obtain this upper bound:
\begingroup\small
\begin{align*}
P_{(d,\phi_d(n))}&\le\frac{9-(9d-1)\ln(\frac{d+1}{d})+10\ln(\frac{10d}{d+1})}{81d}\\
&\enspace+\frac{\frac{9}{d}+90d\ln(\frac{d+1}{d})+72n\ln(\frac{d+1}{d})+
\frac{10\ln(\frac{10d}{d+1})}{d}+90n}{81(d\times10^n-1)}.
\end{align*}
\endgroup
The bound just above and the one brought to light in inequality (\ref{lboui}), added to the following limits: 
\begingroup\scriptsize
\begin{align*}
&\lim\limits_{n\to+\infty}(\frac{90+100\ln(\frac{10d}{d+1})+9\ln(\frac{10d}{d+1})n+81dn}
{81d\times10^n})=0\quad\text{and}\\
&\lim\limits_{n\to+\infty}(\frac{\frac{9}{d}+90d\ln(\frac{d+1}{d})+72n\ln(\frac{d+1}{d})+
\frac{10\ln(\frac{10d}{d+1})}{d}+90n}{81(d\times10^n-1)})=0
\end{align*}
\endgroup
lead to the expected result.
\end{proof}

Let us denote by $\alpha_d$ the limit of $(P_{(d,\phi_d(n))})_{n\in\mathbb N^*}$: 
\begin{align*}
\alpha_d=\frac{9+10\ln10+9(d+1)\ln(1-\frac{1}{d+1})}{81d}\,.
\end{align*}

Here is the first values of $P_{(1,\phi(n))}$ ($\alpha_1\approx0.241$):
\begin{table}[ht]
\centering
\begin{tabular}{|c||c|c|c|}
\hline
\rowcolor{gray!40}$n$&$\phi(n)$&$P_{(i,\phi(n))}$&$P_{(i,\phi(n))}-\alpha_1$\\\hline
$1$& $9$ & $0.314$ & $7.30\times10^{-2}$   \\\hline
$2$& $99$ & $0.253$ & $1.12\times10^{-2}$   \\\hline
$3$& $999$ & $0.243$ & $1.55\times10^{-3}$   \\\hline
$4$& $9999$ & $0.242$ & $1.99\times10^{-4}$   \\\hline
$5$& $99999$ & $0.241$ & $2.43\times10^{-5}$   \\\hline
\end{tabular}
\caption{First five values $P_{(1,\phi(n))}$ and $P_{(1,\phi(n))}-\alpha_1$. We round off these values to three significant digits.}
\label{tab1}
\end{table}

Here is a few values of $P_{(d,\phi_d(n))}$, for $d\in\llbracket2,9\rrbracket$:
\begin{table}[ht]
\centering
\begin{tabular}{|c||c|c|c|c||c|}
\hline
\rowcolor{gray!40}$d$&$P_{(d,\phi_d(1))}$&$P_{(d,\phi_d(2))}$&$P_{(d,\phi_d(3))}$&$P_{(d,\phi_d(4))}$&$\alpha_d$\\\hline
$2$& $0.134$ & $0.131$ & $0.130$ & $0.130$ & $0.130$  \\\hline
$3$& $0.085$ & $0.089$ & $0.089$ & $0.089$ & $0.089$  \\\hline
$4$& $0.062$ & $0.067$ & $0.068$ & $0.068$ & $0.068$  \\\hline
$5$& $0.049$ & $0.054$ & $0.055$ & $0.055$ & $0.055$  \\\hline
$6$& $0.040$ & $0.045$ & $0.046$ & $0.046$ & $0.046$  \\\hline
$7$& $0.034$ & $0.039$ & $0.039$ & $0.040$ & $0.040$  \\\hline
$8$& $0.030$ & $0.034$ & $0.035$ & $0.035$ & $0.035$  \\\hline
$9$& $0.026$ & $0.030$ & $0.031$ & $0.031$ & $0.031$  \\\hline
\end{tabular}
\caption{Values of $P_{(d,\phi_d(n))}$ and $\alpha_d$, for $n\in\llbracket1,4\rrbracket$. These values are rounded to the nearest thousandth.}
\label{tab5}
\end{table}

\subsection{The second subsequence}

The second subsequence we will consider is $(P_{(d,\psi_d(n))})_{n\in\mathbb N^*}$ where $\psi_d$ is the function from $\mathbb N^*$ to $\mathbb N$ that 
maps $n$ to $(d+1)\times10^n-1$. We get the following result:
\begin{proposition}\label{sub2i}
The subsequence $(P_{(d,\psi_d(n))})_{n\in\mathbb N^*}$ converges to:
\begin{align*}
\frac{10\big(9-(9d-1)\ln(\frac{d+1}{d})+\ln(\frac{10d}{d+1})\big)}{81(d+1)}=\frac{10\big(9+\ln10+9d\ln(1-\frac{1}{d+1})\big)}{81(d+1)}\, .
\end{align*}
\end{proposition}

\begin{proof}
Let $n$ be a positive integer such that $n\ge2$. According to Proposition \ref{proi}, we have:
\begingroup\scriptsize
\begin{align*}
P_{(d,\psi_d(n))}&=P_{(d,(d+1)\times10^n-1)}\\
&=\frac{1}{(d+1)\times10^n-1}\Big(\sum_{l=0}^{n}\sum_{b=d\times10^l}^{(d+1)\times10^l-1}\frac{b-\frac{(9d-1)10^l-8}{9}}{b}+\sum_{l=0}^{n-1}
\sum_{a=(d+1)\times10^l}^{d\times10^{l+1}-1}\frac{\frac{10^{l+1}-1}{9}}{a}\Big)\,.
\end{align*}
\endgroup
Let us first find an appropriate lower bound of $P_{(d,\psi_d(n))}$ in a way very similar to that used in the proof of Proposition \ref{sub1i}:
\begingroup\scriptsize
\begin{align*}
P_{(d,\psi_d(n))}&\ge\frac{1}{(d+1)\times10^n}\Big(\frac{10(10^{n}-1)}{9}-\frac{(9d-1)\ln(\frac{d+1}{d})}{9}\times\frac{10(10^n-1)}{9}\\
&\enspace-\frac{9d-1}{9}\sum_{l=1}^{n}10^l\ln(1+\frac{\frac{1}{d(d+1)}}{10^l-\frac{1}{d}})+\frac{\ln(\frac{10d}{d+1})}{9}\times
\big(\frac{10^2(10^{n-1}-1)}{9}-(n-1)\big)\Big)\\
&\ge\frac{90-10(9d-1)\ln(\frac{d+1}{d})+10\ln(\frac{10d}{d+1})}{81(d+1)}-\frac{90+100\ln(\frac{10d}{d+1})+9n\ln(\frac{10d}{d+1})}{81(d+1)\times10^n}\\
&\enspace-d\sum_{l=1}^{n}\frac{10^l\frac{\frac{1}{d+1}}{d\times10^l-1}}{(d+1)\times10^n}\,.
\end{align*}
\endgroup
Then:
\begingroup\scriptsize
\begin{align}\label{lbouni}
P_{(d,\psi_d(n))}&\ge\frac{90-10(9d-1)\ln(\frac{d+1}{d})+10\ln(\frac{10d}{d+1})}{81(d+1)}-
\frac{90+100\ln(\frac{10d}{d+1})+9n\ln(\frac{10d}{d+1})+81dn}{(d+1)\times10^n}\,.
\end{align}
\endgroup
Let us now find an appropriate upper bound of $P_{(d,\psi_d(n))}$ using the proof of Proposition \ref{sub1i}:
\begingroup\scriptsize
\begin{align*}
P_{(d,\psi_d(n))}&\le\frac{1}{(d+1)\times10^n-1}\Big(\sum_{l=0}^n10^l-\frac{\ln(\frac{d+1}{d})}{9}\sum_{l=0}^{n}\big((9d-1)10^l-8\big)\\
&\enspace+\sum_{l=0}^{n-1}\frac{10^{l+1}}{9}\big(\ln(\frac{10d}{d+1})+\ln(1+\frac{\frac{1}{d+1}-\frac{1}{10d}}{10^l-\frac{1}{d+1}})\big)\Big)\\
&\le\frac{1}{(d+1)(10^n-\frac{1}{d+1})}\Big(\frac{10^{n+1}-1}{9}-\frac{\ln(\frac{d+1}{d})}{9}
\big(\frac{(9d-1)(10^{n+1}-1)}{9}-8(n+1))\\
&\enspace+\frac{10\ln(\frac{10d}{d+1})}{9}\times\frac{10^n-1}{9}+\frac{10}{9}\sum_{l=0}^{n-1}\frac{10^l(1-\frac{d+1}{10d})}{(d+1)10^l-1}\Big)\\
&\le\frac{1}{(d+1)(10^n-\frac{1}{d+1})}\Big(\frac{10(10^n-\frac{1}{d+1})-1+\frac{10}{d+1}}{9}+
\frac{10\ln(\frac{10d}{d+1})(10^n-\frac{1}{d+1}+\frac{1}{d+1}-1)}{81}\\
&\enspace-\frac{\ln(\frac{d+1}{d})}{9}\big(\frac{(9d-1)(10(10^n-\frac{1}{d+1})-1+\frac{10}{d+1})}{9}-8(n+1)\big)+\frac{10}{9}\sum_{l=0}^{n-1}1\Big)\,.
\end{align*}
\endgroup
Thereby:
\begingroup\scriptsize
\begin{align}\label{ubouni}
P_{(d,\psi_d(n))}&\le\frac{90-10(9d-1)\ln(\frac{d+1}{d})+10\ln(\frac{10d}{d+1})}{81(d+1)}+\frac{\frac{90}{d+1}+72(n+1)\ln(\frac{d+1}{d})+90n}
{81\big((d+1)\times10^n-1\big)}\,.
\end{align}
\endgroup
Bounds brought to light in inequalities (\ref{lbouni}) and (\ref{ubouni}) and the fact that:
\begin{align*}
&\lim\limits_{n\to+\infty}(\frac{90+100\ln(\frac{10d}{d+1})+9n\ln(\frac{10d}{d+1})+81dn}{(d+1)\times10^n})=0
\quad\text{and}\\
&\lim\limits_{n\to+\infty}(\frac{\frac{90}{d+1}+72(n+1)\ln(\frac{d+1}{d})+90n}{81\big((d+1)\times10^n-1\big)})=0
\end{align*}
lead to the expected result.
\end{proof}

Let us denote by $\beta_d$ the limit of $(P_{(d,\psi_d(n))})_{n\in\mathbb N^*}$: 
\begin{align*}
\beta_d=\frac{10\big(9+\ln10+9d\ln(1-\frac{1}{d+1})\big)}{81(d+1)}\,.
\end{align*}

Here is the first values of $P_{(1,\psi(n))}$ ($\beta_1\approx0.313$):
\begin{table}[ht]
\centering
\begin{tabular}{|c||c|c|c|}
\hline
\rowcolor{gray!40}$n$&$\psi(n)$&$P_{(i,\psi(n))}$&$P_{(i,\psi(n))}-\beta_1$\\\hline
$1$& $19$ & $0.373$ & $6.00\times10^{-2}$   \\\hline
$2$& $199$ & $0.321$ & $7.93\times10^{-3}$   \\\hline
$3$& $1999$ & $0.314$ & $1.01\times10^{-3}$   \\\hline
$4$& $19999$ & $0.313$ & $1.23\times10^{-4}$   \\\hline
$5$& $199999$ & $0.313$ & $1.45\times10^{-5}$   \\\hline
\end{tabular}
\caption{First five values $P_{(1,\psi(n))}$ and $P_{(1,\psi(n))}-\beta_1$. We round off these values to three significant digits.}
\label{tab2}
\end{table}

Here is a few values of $P_{(d,\psi_d(n))}$, for $d\in\llbracket2,9\rrbracket$:
\begin{table}[ht]
\centering
\begin{tabular}{|c||c|c|c|c||c|}
\hline
\rowcolor{gray!40}$d$&$P_{(d,\psi_d(1))}$&$P_{(d,\psi_d(2))}$&$P_{(d,\psi_d(3))}$&$P_{(d,\psi_d(4))}$&$\beta_d$\\\hline
$2$& $0.176$ & $0.166$ & $0.165$ & $0.165$ & $0.165$  \\\hline
$3$& $0.110$ & $0.109$ & $0.109$ & $0.109$ & $0.109$  \\\hline
$4$& $0.078$ & $0.080$ & $0.081$ & $0.0.081$ & $0.081$  \\\hline
$5$& $0.060$ & $0.063$ & $0.064$ & $0.064$ & $0.064$  \\\hline
$6$& $0.049$ & $0.052$ & $0.052$ & $0.053$ & $0.053$  \\\hline
$7$& $0.041$ & $0.044$ & $0.045$ & $0.045$ & $0.045$  \\\hline
$8$& $0.035$ & $0.038$ & $0.039$ & $0.039$ & $0.039$  \\\hline
$9$& $0.031$ & $0.034$ & $0.034$ & $0.034$ & $0.034$  \\\hline
\end{tabular}
\caption{Values of $P_{(d,\psi_d(n))}$ and $\beta_d$, for $n\in\llbracket1,4\rrbracket$. These values are rounded to the nearest thousandth.}
\label{tab6}
\end{table}

\section{The graph of \texorpdfstring{$(P_{(d,n)})_{n\in\mathbb N^*}$}{}}

Let us first plot the graph of the sequence $(P_{(1,n)})_{n\in\mathbb N^*}$ for values of $n$ from $1$ to $1200$ (Figure \ref{fig1}). Then we plot a 
second graph of $P_{(1,n)}$ \textit{versus} $\log(n)$, for $n\in\llbracket1,32000\rrbracket$ (Figure \ref{fig2}). On this graph, the four dots 
represented by red circles are associated with the first values of $(P_{(1,\phi(n))})_{n\in\mathbb N^*}$ and the blue ones are associated with the 
subsequence $(P_{(1,\psi(n))})_{n\in\mathbb N^*}$. Their distances to the same-coloured horizontal dotted asymptote tend towards $0$. The linear 
equations of these lines are $y=\alpha_1$ and $y=\beta_1$ respectively (see Propositions \ref{sub1i} and \ref{sub2i}).

\begin{figure}[ht]
\begin{minipage}{0.47\linewidth}
\centering
\includegraphics[scale=0.44,clip=true]{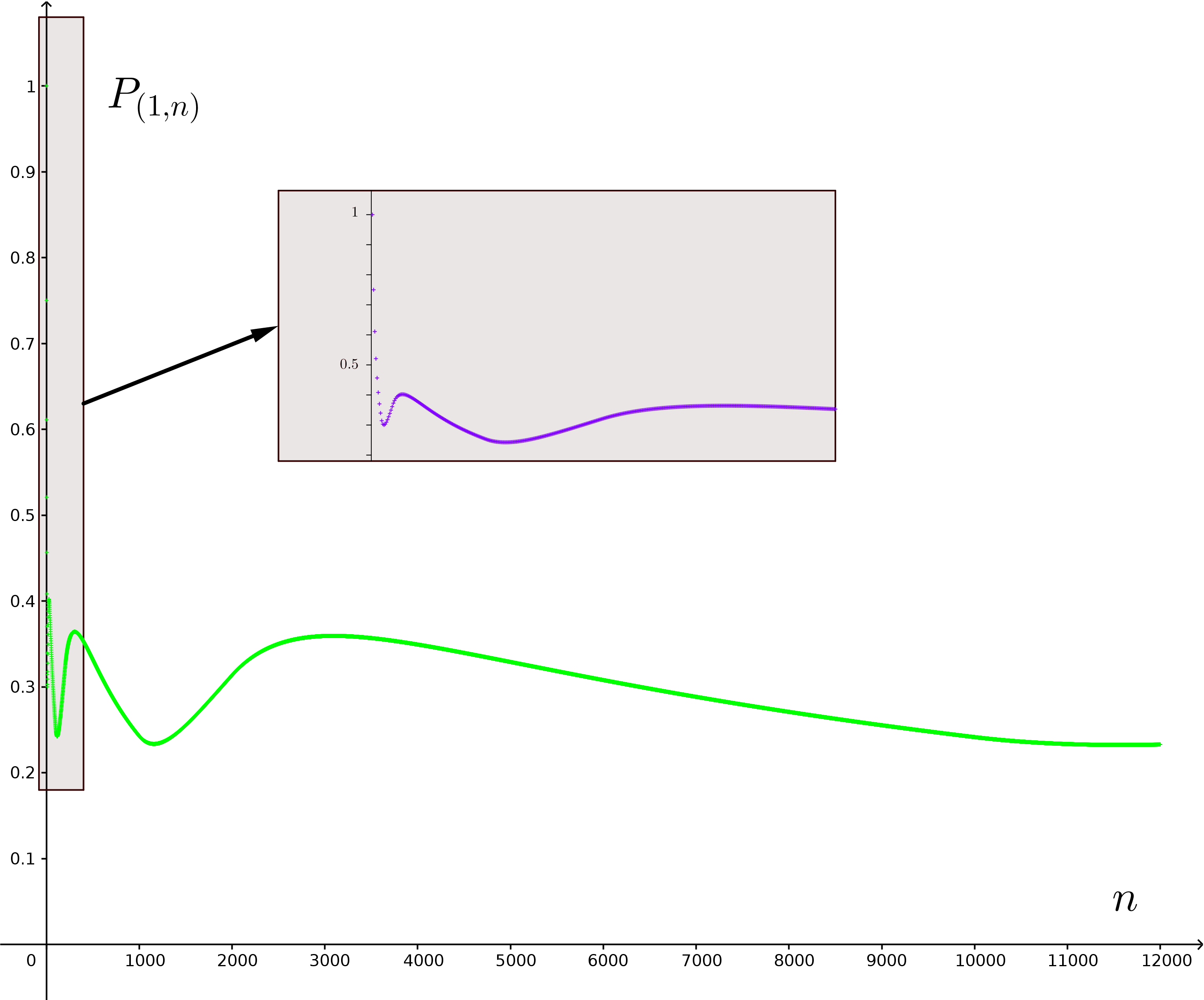}
\caption{Graph of $(P_{(1,n)})_{n\in\mathbb N^*}$.}\label{fig1}
\vspace{4.94\baselineskip}

\end{minipage}\hfill
\begin{minipage}{0.49\linewidth}
\centering
\includegraphics[scale=0.47,clip=true]{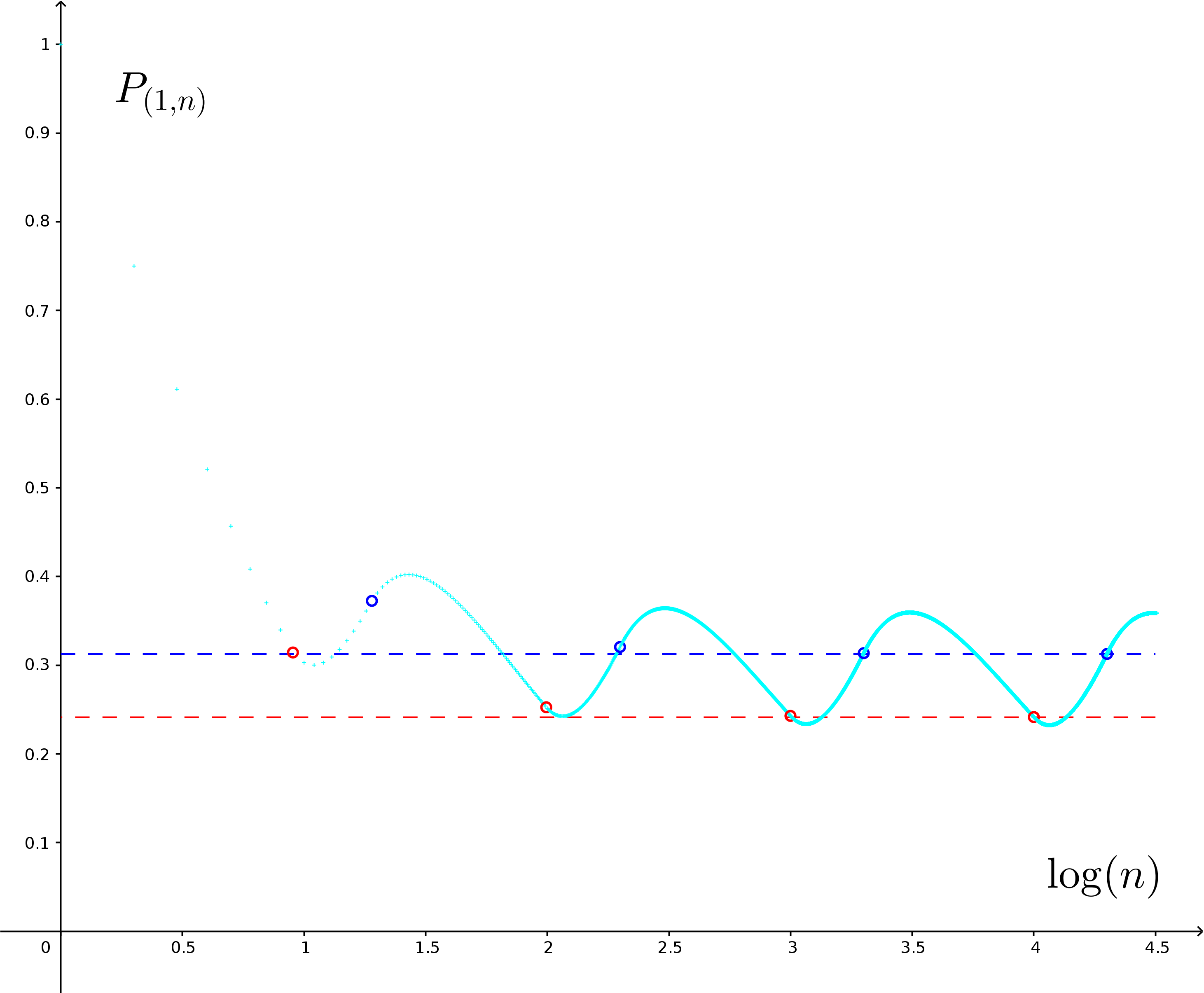}
\caption{Graph of $P_{(1,n)}$ \textit{versus} $\log(n)$, values of $P_{(1,\phi(n))}$ being in red and those of $P_{(1,\psi(n))}$ being in deep blue. 
Limits $\alpha_1$ and $\beta_1$ of these two subsequences are represented by horizontal asymptotes.}\label{fig2}
\end{minipage}
\end{figure}

Through Figure \ref{fig2}, it is clear that the proportion of $1$ as leading digit structurally fluctuate and does not follow Benford's Law.

Let us additionally plot graphs of sequences $(P_{(d,n)})_{n\in\mathbb N^*}$ for values of $n$ from $1$ to $400$ 
(Figure \ref{fig3}). Then we plot graphs of $P_{(d,n)}$ \textit{versus} $\log(n)$, for $n\in\llbracket1,32000\rrbracket$ (Figure \ref{fig4}).

\begin{figure}[ht]
\begin{minipage}{0.48\linewidth}
\centering
\includegraphics[scale=0.5,clip=true]{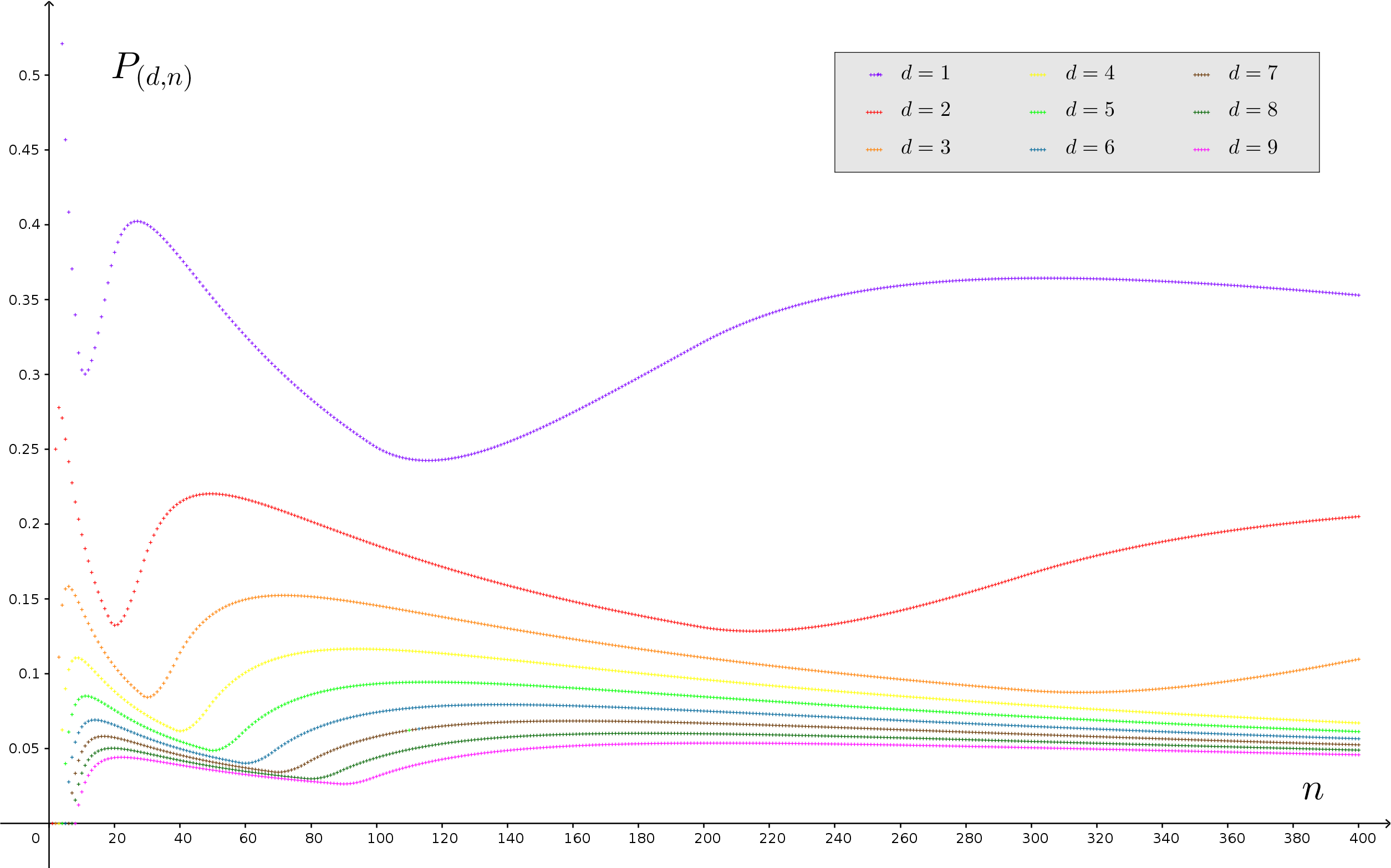}
\caption{Graphs of $(P_{(d,n)})_{n\in\mathbb N^*}$, for $d\in\llbracket1,9\rrbracket$.}\label{fig3}
\vspace{0.695\baselineskip}

\end{minipage}\hfill
\begin{minipage}{0.48\linewidth}
\centering
\includegraphics[scale=0.55,clip=true]{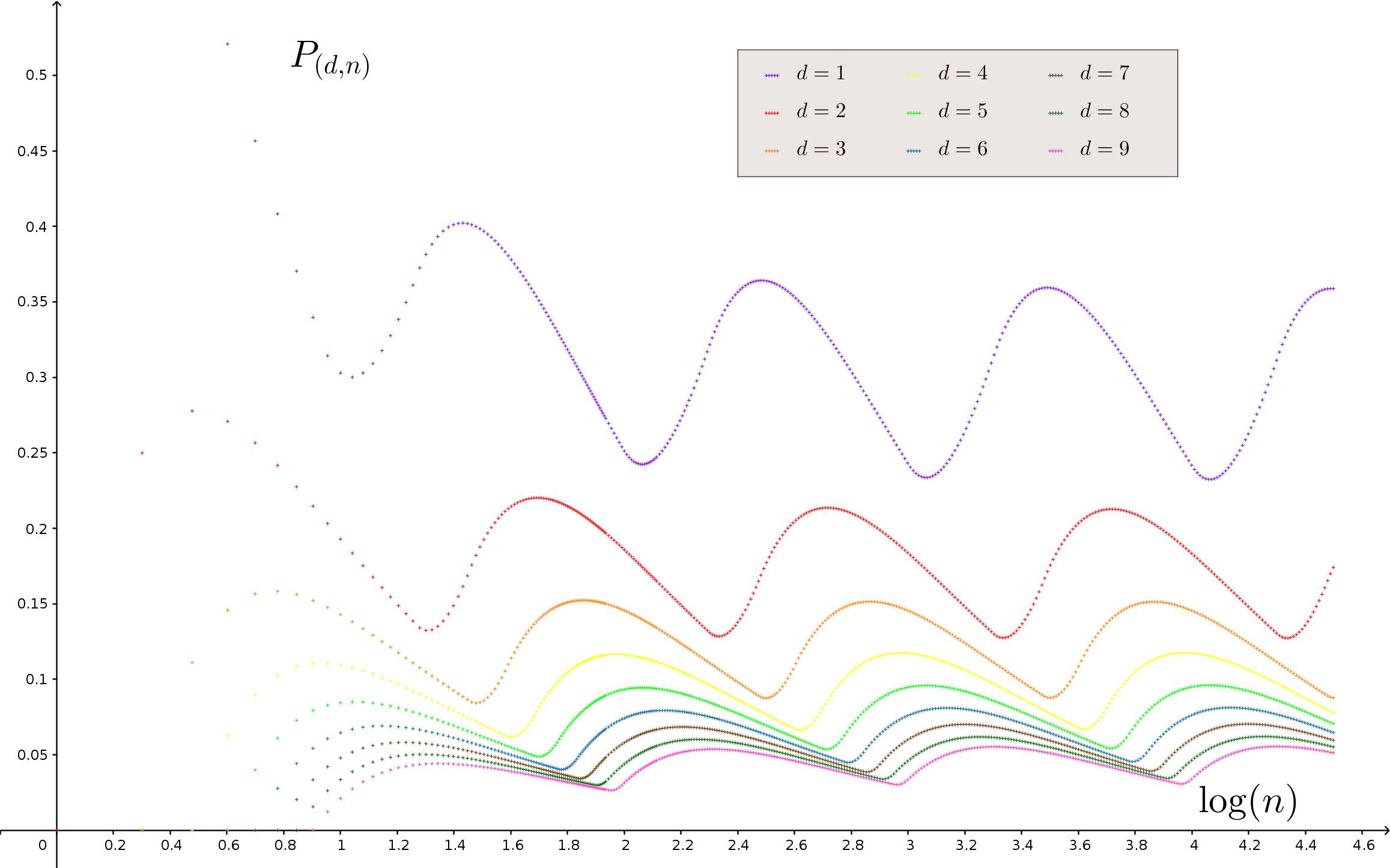}
\caption{For $d\in\llbracket1,9\rrbracket$, graphs of $P_{(d,n)}$ \textit{versus} $\log(n)$. Note that points have not been all represented.}\label{fig4}
\end{minipage}
\end{figure}

Through Figure \ref{fig4}, it is once more clear that the proportion of each $d$ as leading digit, $d\in\llbracket1,9\rrbracket$, structurally fluctuate 
and does not follow Benford's Law.

For each $d\in\llbracket1,9\rrbracket$ the values seem to fluctuate between two values, under the following constraint:

\begin{proposition}
For all $n\in\mathbb{N}^*$ such that $n\ge10$ and for all $(p,q)\in\llbracket1,9\rrbracket^2$ such that $p<q$, we have:
\begin{align*}
P_{(p,n)}>P_{(q,n)}\,.
\end{align*}
\end{proposition}

The relative position of graphs of $P_{(d,n)}$, for $d\in\llbracket1,9\rrbracket$, can be observed on Figures \ref{fig3} and \ref{fig4}.

\begin{proof}
For all $n\in\mathbb{N}^*$ such that $n\ge10$ and for all $(p,q)\in\llbracket1,9\rrbracket^2$ such that $p<q$, we denote by $k_p$ and $k_q$ 
the positive integers such that $k_p=\min\{i\in\mathbb N:p\times10^i>n\}$ and $k_q=\min\{i\in\mathbb N:q\times10^i>n\}$. We note that $k_p\ge k_q$. 
For $(r,l)\in\mathbb{N}^2$, let $A_{l}$ and $B_{r,l}$ be real numbers such that $A_{l}=\frac{10^{l+1}-1}{9}$ and 
$B_{r,l}=\frac{(9r-1)\times10^{l}-8}{9}$. We consider that $A_{-1}=0$. Third cases can be distinguished:

In the first case $(q+1)\times10^{k_p-2}\le p\times10^{k_p-1}\le n<(p+1)\times10^{k_p-1}\le q\times10^{k_p-1}$. Thus we have $k_p=k_q+1$. Thanks to 
Proposition \ref{proi} we obtain:
\begingroup\scriptsize
\begin{align*}
P_{(p,n)}-P_{(q,n)}&=\frac{1}{n}\Big(\sum_{l=0}^{k_p-2}(\sum_{b=p\times10^l}^{(p+1)\times10^l-1}\frac{b-B_{p,l}}{b}+
\sum_{a=(p+1)\times10^l}^{p\times10^{l+1}-1}\frac{A_{l}}{a})+\sum_{b=p\times10^{k_p-1}}^{n}\frac{b-B_{p,k_p-1}}{b}\Big)\\
&\enspace-\frac{1}{n}\Big(\sum_{l=0}^{k_p-3}(\sum_{b=q\times10^l}^{(q+1)\times10^l-1}\frac{b-B_{q,l}}{b}
+\sum_{a=(q+1)\times10^l}^{q\times10^{l+1}-1}\frac{A_{l}}{a})\\
&\enspace+\sum_{b=q\times10^{k_p-2}}^{(q+1)\times10^{k_p-2}-1}\frac{b-B_{q,k_p-2}}{b}+\sum_{a=(q+1)\times10^{k_p-2}}^{n}\frac{A_{k_p-2}}{a}\Big)\\
&=\frac{1}{n}\Big(\sum_{l=0}^{k_p-3}\big(\sum_{b=p\times10^l}^{(p+1)\times10^l-1}\frac{b-B_{p,l}}{b}+
\sum_{a=(p+1)\times10^l}^{p\times10^{l+1}-1}\frac{A_{l}}{a}-\sum_{b=q\times10^l}^{(q+1)\times10^l-1}\frac{b-B_{q,l}}{b}\\
&\enspace-\sum_{a=(q+1)\times10^l}^{q\times10^{l+1}-1}\frac{A_{l}}{a}\big)+\sum_{b=p\times10^{k_p-2}}^{(p+1)\times10^{k_p-2}-1}\frac{b-B_{p,{k_p-2}}}{b}+
\sum_{a=(p+1)\times10^{k_p-2}}^{p\times10^{k_p-1}-1}\frac{A_{k_p-2}}{a}\\
&\enspace+\sum_{b=p\times10^{k_p-1}}^{n}\frac{b-B_{p,k_p}}{b}-
\sum_{b=q\times10^{k_p-2}}^{(q+1)\times10^{k_p-2}-1}\frac{b-B_{q,k_p}}{b}\\
&\enspace-\sum_{a=(q+1)\times10^{k_p-2}}^{n}\frac{A_{k_p-2}}{a}\Big)\\
&=\frac{1}{n}\Big(\sum_{l=0}^{k_p-2}\big(\sum_{b=p\times10^l}^{(p+1)\times10^l-1}\frac{b-B_{p,l}-A_{l-1}}{b}+
\sum_{a=(p+1)\times10^l}^{q\times10^l-1}\frac{A_{l}-A_{l-1}}{a}\\
&\enspace+\sum_{b=q\times10^l}^{(q+1)\times10^l-1}\frac{A_l-(b-B_{q,l})}{b}+\sum_{a=(q+1)\times10^l}^{p\times10^{l+1}-1}\frac{A_{l}-A_{l}}{a}\big)\\
&\enspace+\sum_{b=p\times10^{k_p-1}}^{n}\frac{b-B_{p,k_p-1}-A_{k_p-2}}{b}\Big)\,.
\end{align*}
\endgroup
Furthermore, for all $(r,l)\in\mathbb{N}^2$, we have $A_l>A_{l-1}$ and:
\begin{align*}\forall s\in\llbracket r\times10^l,(r+1)\times10^{l+1}-1&\rrbracket,\\
s-B_{r,l}-A_{l-1}&=s-\frac{(9r-1)\times10^{l}-8}{9}-\frac{10^l-1}{9}\\
&=s-(r\times10^{l}-1)>0\\
A_l-(s-B_{r,l})&=\frac{10^{l+1}-1}{9}-s+\frac{(9r-1)\times10^{l}-8}{9}\\
&=(r+1)\times10^l-1-s\ge0\,.
\end{align*}
Consequently, in this case, $P_{(p,n)}-P_{(q,n)}>0$.

In the second case $(p+1)\times10^{k_p-1}\le n<q\times10^{k_p-1}$. Thus we have $k_p=k_q+1$. Thanks to Proposition \ref{proi} we obtain:
\begingroup\scriptsize
\begin{align*}
P_{(p,n)}-P_{(q,n)}&=\frac{1}{n}\Big(\sum_{l=0}^{k_p-2}(\sum_{b=p\times10^l}^{(p+1)\times10^l-1}\frac{b-B_{p,l}}{b}+
\sum_{a=(p+1)\times10^l}^{p\times10^{l+1}-1}\frac{A_{l}}{a})\\
&\enspace+\sum_{b=p\times10^{k_p-1}}^{(p+1)\times10^{k_p-1}-1}\frac{b-B_{p,k_p-1}}{b}+\sum_{a=(p+1)\times10^{k_p-1}}^{n}\frac{A_{k_p-1}}{a}\Big)
-\frac{1}{n}\Big(\sum_{l=0}^{k_p-3}\\
&\enspace(\sum_{b=q\times10^l}^{(q+1)\times10^l-1}\frac{b-B_{q,l}}{b}+\sum_{a=(q+1)\times10^l}^{q\times10^{l+1}-1}\frac{A_{l}}{a})+
\sum_{b=q\times10^{k_p-2}}^{(q+1)\times10^{k_p-2}-1}\frac{b-B_{q,k_p-2}}{b}\\
&\enspace+\sum_{a=(q+1)\times10^{k_p-2}}^{n}\frac{A_{k_p-2}}{a}\Big)\\
&=\frac{1}{n}\Big(\sum_{l=0}^{k_p-2}\big(\sum_{b=p\times10^l}^{(p+1)\times10^l-1}\frac{b-B_{p,l}-A_{l-1}}{b}+
\sum_{a=(p+1)\times10^l}^{q\times10^l-1}\frac{A_{l}-A_{l-1}}{a}\\
&\enspace+\sum_{b=q\times10^l}^{(q+1)\times10^l-1}\frac{A_l-(b-B_{q,l})}{b}\big)
+\sum_{b=p\times10^{k_p-1}}^{(p+1)\times10^{k_p-1}-1}\frac{b-B_{p,k_p-1}-A_{k_p-2}}{b}\\
&\enspace+\sum_{a=(p+1)\times10^{k_p-1}}^{n}\frac{A_{k_p-1}-A_{k_p-2}}{a}\Big)\,.
\end{align*}
\endgroup
Consequently, in this case, $P_{(p,n)}-P_{(q,n)}>0$.

In the third case $q\times10^{k_p-1}\le n<p\times10^{k_p}$. Thus we have $k_p=k_q$. Thanks to Proposition \ref{proi} we obtain:
\begingroup\scriptsize
\begin{align*}
P_{(p,n)}-P_{(q,n)}
&=\frac{1}{n}\Big(\sum_{l=0}^{k_p-2}\big(\sum_{b=p\times10^l}^{(p+1)\times10^l-1}\frac{b-B_{p,l}-A_{l-1}}{b}+
\sum_{a=(p+1)\times10^l}^{q\times10^l-1}\frac{A_{l}-A_{l-1}}{a}\\
&\enspace+\sum_{b=q\times10^l}^{(q+1)\times10^l-1}\frac{A_l-(b-B_{q,l})}{b}\big)
+\sum_{b=p\times10^{k_p-1}}^{(p+1)\times10^{k_p-1}-1}\frac{b-B_{p,k_p-1}-A_{k_p-2}}{b}\\
&\enspace+\sum_{a=(p+1)\times10^{k_p-1}}^{q\times10^{k_p-1}}\frac{A_{k_p-1}-A_{k_p-2}}{a}+
\sum_{b=q\times10^{k_p-1}}^{\min((q+1)\times10^{k_p-1}-1,n)}\frac{A_l-(b-B_{q,l})}{b}\Big)\,.
\end{align*}
\endgroup
Consequently, in this latter case, $P_{(p,n)}-P_{(q,n)}>0$.
\end{proof}

\begin{remark}
For $n\in\mathbb{N}^*$, we have, if $d>n$, $P_{(d,n)}=0$. Hence for all $n\in\mathbb N^*$ and for all 
$(p,q)\in\llbracket1,9\rrbracket^2$ such that $p<q$, we have:
\begin{align*}
P_{(p,n)}\ge P_{(q,n)}\,.
\end{align*}
\end{remark}

Let us from now on denote by $k_d$ the positive integer such that $k_d=\min\{i\in\mathbb N:d\times10^i>n\}$. Through Figures \ref{fig2} and \ref{fig4}, 
"delayed effects" appear to exist, in particular after a long sequence of numbers whose leading digit is $d$. Let us examine these features in more 
detail. We study for this purpose the increasing or decreasing nature of sequences $(P_{(d,n)})_{n\in\mathbb N^*}$, for $d\in\llbracket1,9\rrbracket$:

\begin{proposition}$i\in\{0,1\}$.
\begin{align*}
P_{(d,n+1)}-P_{(d,n)}=\frac{1}{n+1}(\frac{c_i}{n+1}-P_{(d,n)}),
\end{align*}
where:
\begingroup\small
\[c_i=
\left \{
\begin{array}{l @{\quad} l}
    n+1-\frac{(9d-1)10^{k_d}-8}{9} & \text{if }\quad  n+1\in\llbracket d\times10^{k_d},(d+1)\times10^{k_d}
    -1\rrbracket\\
    \frac{10^{k_d+1}-1}{9} & \text{if }\quad n+1\in\llbracket(d+1)\times10^{k_d},d\times10^{k_d+1}
    -1\rrbracket
\end{array}
\right.\]
\endgroup
\end{proposition}

\begin{proof}
It is based on the formulas of Proposition \ref{proi}. Indeed:
\begin{align*}
P_{(d,n+1)}=\frac{1}{n+1}(nP_{(d,n)}+\frac{c_i}{n+1})=P_{(d,n)}+\frac{1}{n+1}(\frac{c_i}{n+1}-P_{(d,n)}).
\end{align*}
\end{proof}

Through this proposition the obvious condition regarding the increasing or decreasing nature of the sequence is underscored: whether $P_{(d,n)}$ value, 
for $d\in\llbracket1,9\rrbracket$, is less or greater than the appropriate $\frac{c_i}{n+1}$ value, for $n\in\mathbb N^*$ and $i\in\{0,1\}$. Finding the 
approximate values of $n$ for which the increasing or decreasing nature of the sequence $(P_{(d,n)})_{n\in\mathbb N^*}$ appears is henceforth the aim of 
that section. We first provide a proposition similar to the previous one:

\begin{proposition}\label{pro2i}
If $n\in\llbracket d\times10^{k_d},(d+1)\times10^{k_d}-1\rrbracket$, then we have:
\begin{align*}
P_{(d,n)}=\frac{1}{n}\Big(P_{(d,d\times10^{k_d}-1)}\times(d\times10^{k_d}-1)+\sum_{b=d\times10^{k_d}}^n\frac{b-\frac{(9d-1)10^{k_d}-8}{9}}{b}\Big)\,.
\end{align*}

If $n\in\llbracket(d+1)\times10^{k_d},d\times10^{k_d+1}-1\rrbracket$, then we have:
\begin{align*}
P_{(d,n)}=\frac{1}{n}\Big(P_{(d,(d+1)\times10^{k_d}-1)}\times((d+1)\times10^{k_d}-1)+\sum_{a=(d+1)\times10^{k_d}}^n\frac{\frac{10^{k_d+1}-1}{9}}{a}\Big)\,.
\end{align*}
\end{proposition}

\begin{proof}
Results are directly derived from Proposition \ref{proi}.
\end{proof}

We now consider the sequence $(\widehat P_{(d,n)})_{n\in\mathbb N^*}$ defined as follows.\\
If $n\in\llbracket d\times10^{k_d},(d+1)\times10^{k_d}-1\rrbracket$:
\begin{align*}
\widehat P_{(d,n)}=d(\alpha_d-1)\frac{10^{k_d}}{n}+1+\frac{9d-1}{9}\frac{10^{k_d}}{n}\big(\ln(\frac{10^{k_d}}{n})+\ln d\big)\,.
\end{align*}
If $n\in\llbracket(d+1)\times10^{k_d},d\times10^{k_d+1}-1\rrbracket$:
\begin{align*}
\widehat P_{(d,n)}=(d+1)\beta_d\frac{10^{k_d}}{n}-\frac{10}{9}\frac{10^{k_d}}{n}\big(\ln(\frac{10^{k_d}}{n})+\ln(d+1)\big)\,.
\end{align*}

We denote by $\gamma_d$ a real number such that $\gamma_d\in]1;1+\frac{1}{d}[$ and $I_{(d,\gamma_d)}$ the set such that 
$I_{(d,\gamma_d)}=\bigcup_{i=1}^{+\infty}\big(\llbracket d\gamma_d\times10^{i},(d+1)\times10^{i}-1\rrbracket\cup\llbracket(d+1)\gamma_d\times10^{i},
d\times10^{i+1}-1\rrbracket\big)$. The below proposition can thereupon be stated:
\begin{proposition}\label{pro3i}
\begin{align*}
P_{(d,n)}\underset{\substack{n \to +\infty \\ n\in I_{(d,\gamma_d)}}}{\sim}\widehat P_{(d,n)}\,.
\end{align*}
\end{proposition}

\begin{proof}Let us study both cases. In the first one, $n\in\llbracket\gamma_d d\times10^{k_d},(d+1)\times10^{k_d}-1\rrbracket$. Let $I_1$ be the 
interval such that, $I_1=\bigcup_{i=1}^{+\infty}\llbracket d\gamma_d\times10^{i},(d+1)\times10^{i}-1\rrbracket$. We have:
\begin{align}\label{ine1i}
\frac{1}{d+1}\le\frac{10^{k_d}}{n}\le\frac{1}{\gamma_dd}\,.
\end{align}
Before we go any further, let us prove the following lemma:
\begin{lemma}\label{lem0i}
For $n>d\times10^{k_d}$:
\begingroup\small
\begin{align*}
\ln(\frac{n}{10^{k_d}})-\ln d+\ln(1+\frac{1}{n})\le\sum_{b=d\times10^{k_d}}^n\frac{1}{b}\le\ln(\frac{n}{10^{k_d}})-\ln d+
\ln(1+\frac{1}{d\times10^{k_d}-1})\,.
\end{align*}
\endgroup
\end{lemma}
\begin{proof}
This result is directly related to inequalities \ref{inelog}. Indeed for $n>d\times10^{k_d}>1$, we have:
\begin{align*}
\ln(\frac{n+1}{d\times10^{k_d}})\le\sum_{b=d\times10^{k_d}}^n\frac{1}{b}\le\ln(\frac{n}{d\times10^{k_d}-1})\,.
\end{align*}
\end{proof}
Thanks to Proposition \ref{pro2i} we get:
\begingroup\scriptsize
\begin{align*}
P_{(d,n)}&=P_{(d,d\times10^{k_d}-1)}\frac{d\times10^{k_d}-1}{n}+\frac{n-d\times10^{k_d}+1}{n}-\frac{(9d-1)10^{k_d}-8}{9n}
\sum_{b=d\times10^{k_d}}^n\frac{1}{b}\\
&=d(P_{(d,d\times10^{k_d}-1)}-1)\frac{10^{k_d}}{n}+1-\frac{(9d-1)10^{k_d}}{9n}\sum_{b=d\times10^{k_d}}^n\frac{1}{b}\\
&\enspace-\frac{P_{(1,d\times10^{k_d}-1)}}{n}+\frac{1}{n}+\frac{8}{9n}\sum_{b=d\times10^{k_d}}^n\frac{1}{b}\,.
\end{align*}
\endgroup
Thanks to inequalities (\ref{ine1i}) and knowing that $\sum_{b=d\times10^{k_d}}^n\frac{1}{b}=\ln(\frac{n}{10^{k_d}})-\ln d+
\underset{\substack{n \to +\infty \\ n\in I_1}}{o}(1)$ (thanks to Lemma \ref{lem0i}, $\frac{n}{d\times10^{k_d}}$ being greater than or equal to 
$\gamma_d>1$) and that $P_{(d,d\times10^{k_d}-1)}\underset{+\infty}{\sim}\alpha_d$ (see Proposition \ref{sub1i}), we finally have:
\begin{align*}
P_{(d,n)}\underset{\substack{n \to +\infty \\ n\in I_1}}{\sim}d(\alpha_d-1)\frac{10^{k_d}}{n}+1+\frac{9d-1}{9}\frac{10^{k_d}}{n}
\big(\ln(\frac{10^{k_d}}{n})+\ln d\big)\,.
\end{align*}
In the second case $n\in\llbracket(d+1)\gamma_d\times10^{k_d},d\times10^{{k_d}+1}-1\rrbracket$. Let $I_2$ be the interval such that, 
$I_2=\bigcup_{i=1}^{+\infty}\llbracket(d+1)\gamma_d\times10^{i},d\times10^{i+1}-1\rrbracket$. We have:
\begin{align}\label{ine2i}
\frac{1}{10d}\le\frac{10^{k_d}}{n}\le\frac{1}{(d+1)\gamma_d}\,.
\end{align}
Before we go any further, let us prove this additional lemma:
\begin{lemma}\label{lem0'i}
For $n>(d+1)\times10^{k_d}$:
\begingroup\small
\begin{align*}
\ln(\frac{n}{10^{k_d}})-\ln(d+1)+\ln(1+\frac{1}{n})\le\sum_{a=(d+1)\times10^{k_d}}^n\frac{1}{a}\qquad\text{ and,}
\end{align*}
\begin{align*}
\sum_{a=(d+1)\times10^{k_d}}^n\frac{1}{a}\le\ln(\frac{n}{10^{k_d}})-\ln(d+1)
+\ln(1+\frac{1}{(d+1)\times10^{k_d}-1})\,. 
\end{align*}
\endgroup
\end{lemma}
\begin{proof}
This result is directly related to inequalities \ref{inelog}. \\
Indeed for $n>(d+1)\times10^{k_d}>1$, we have:
\begin{align*}
\ln(\frac{n+1}{(d+1)\times10^{k_d}})\le\sum_{a=(d+1)\times10^{k_d}}^n\frac{1}{a}\le\ln(\frac{n}{(d+1)\times10^{k_d}-1})\,.
\end{align*}
\end{proof}
Thanks to Proposition \ref{pro2i} we get:
\begingroup\footnotesize
\begin{align*}
P_{(d,n)}&=P_{(d,(d+1)\times10^{k_d}-1)}\frac{(d+1)\times10^{k_d}-1}{n}+\frac{1}{9}\times\frac{10^{{k_d}+1}-1}{n}\sum_{a=(d+1)\times10^{k_d}}^n
\frac{1}{a}\\
&=(d+1)P_{(d,(d+1)\times10^{k_d}-1)}\frac{10^{k_d}}{n}+\frac{10}{9}\times\frac{10^{k_d}}{n}\sum_{a=(d+1)\times10^{k_d}}^n\frac{1}{a}\\
&\enspace-\frac{P_{(1,(d+1)\times10^{k_d}-1)}}{n}-\frac{1}{9n}\sum_{a=(d+1)\times10^{k_d}}^n\frac{1}{a}\,.
\end{align*}
\endgroup
Thanks the inequalities (\ref{ine2i}) and knowing that $\sum_{a=(d+1)\times10^{k_d}}^n\frac{1}{a}=\ln(\frac{n}{10^{k_d}})-\ln(d+1)+
\underset{\substack{n \to +\infty \\ n\in I_2}}{o}(1)$ (thanks to Lemma \ref{lem0'i}, $\frac{n}{(d+1)\times10^{k_d}}$ being greater than or equal to 
$\gamma_d>1$) and that $P_{(d,(d+1)\times10^{k_d}-1)}\underset{+\infty}{\sim}\beta_d$ (see Proposition \ref{sub2i}), we finally have:
\begin{align*}
P_{(d,n)}\underset{\substack{n \to +\infty \\ n\in I_2}}{\sim}(d+1)\beta_d\frac{10^{k_d}}{n}-\frac{10}{9}\frac{10^{k_d}}{n}\big(\ln(\frac{10^{k_d}}{n})
+\ln(d+1)\big)\,.
\end{align*}
\end{proof}

To find the approximate values of $n$ for which the increasing or decreasing nature of $(P_{(d,n)})_{n\in\mathbb N^*}$ appears, we need to study two 
distinct functions:

\begin{lemma}\label{lem1i}The minimum $m_d$ of the function $f_d$ from $[\frac{1}{d+1},\frac{1}{d}]$ that maps $x$ onto 
$d(\alpha_d-1)x+1+\frac{9d-1}{9}x(\ln(x)+\ln d)$ is reached when:
\begin{align*}
x=\frac{10^{-\frac{10}{9(9d-1)}}\big(1-\frac{1}{d+1}\big)^{-\frac{d+1}{9d-1}}}{d}\,.
\end{align*}
Its value is $m_d=1-\frac{(9d-1)10^{-\frac{10}{9(9d-1)}}\big(1-\frac{1}{d+1}\big)^{-\frac{d+1}{9d-1}}}{9d}$.
\end{lemma}

\begin{proof}
$\forall x\in[\frac{1}{d+1},\frac{1}{d}]$, $f_d'(x)=d(\alpha_d-1)+\frac{9d-1}{9}(1+\ln(x)+\ln d)$. By solving the equation 
$d(\alpha_d-1)+\frac{9d-1}{9}(1+\ln(x)+\ln d)=0$, we have:
\begin{align*}
x=\frac{10^{-\frac{10}{9(9d-1)}}\big(1-\frac{1}{d+1}\big)^{-\frac{d+1}{9d-1}}}{d}\,.
\end{align*}
Finally $f_d(\frac{10^{-\frac{10}{9(9d-1)}}\big(1-\frac{1}{d+1}\big)^{-\frac{d+1}{9d-1}}}{d})=
1-\frac{(9d-1)10^{-\frac{10}{9(9d-1)}}\big(1-\frac{1}{d+1}\big)^{-\frac{d+1}{9d-1}}}{9d}$.
\end{proof}

\begin{remark}
We note that: 
\begin{align*}
f_d(\frac{1}{d})&=\alpha_d-1+1=\alpha_d\qquad\text{and}\\
f_d(\frac{1}{d+1})&=(\alpha_d-1)\times\frac{d}{d+1}+1+\frac{9d-1}{9}\times\frac{1}{d+1}\ln(\frac{d}{d+1})\\
&=\frac{9+10\ln10+9(d+1)\ln(1-\frac{1}{d+1})-81d}{81(d+1)}+1+\frac{9d-1}{9(d+1)}\ln(\frac{d}{d+1})\\
&=\frac{10\big(9+\ln10+9d\ln(1-\frac{1}{d+1})\big)}{81(d+1)}=\beta_d\,.
\end{align*}
\end{remark}

Using the approximation of Proposition \ref{pro3i}, the definition of 
$(\widehat P_{(d,n)})_{n\in\mathbb N^*}$ (for $n\in\llbracket d\times10^{k_d},(d+1)\times10^{k_d}-1\rrbracket$) and the properties of $f_d$ 
(Lemma \ref{lem1i}), we can approximate the ranks of local minima. Let $i$ be a strictly positive integer. Let us denote by $m_{(d,i)}$ the rank of the 
local minimum of $P_{(d,n)}$ in $\llbracket d\times10^i,(d+1)\times10^i-1\rrbracket$ and by $\widehat m_{(d,i)}$ the estimate of this rank. Let us 
gather in the following table the first values of those ranks and values of local minima, for $d=1$.

\begin{table}[ht]
\centering
\begin{tabular}{|c||c|c|c|}
\hline
\rowcolor{gray!40}$i$&$\widehat m_{(1,i)}$&$m_{(1,i)}$&$P_{(1,m_{(1,i)})}$\\\hline
$1$& $12$ & $11$ & $0.300$   \\\hline
$2$& $116$ & $116$ & $0.242$   \\\hline
$3$& $1158$ & $1158$ & $0.234$   \\\hline
$4$& $11578$ & $11579$ & $0.232$   \\\hline
\end{tabular}
\caption{First four values of above defined ranks and associated values of local minima of $P_{(1,n)}$. We round off values of $P_{(1,n)}$ to three 
significant digits and values of $\widehat m_{(1,i)}$ to unity. Note that $m_1\approx0.232$.}
\label{tab3}
\end{table}

Indeed, for $n\in\llbracket10,19\rrbracket$, our approximation of the rank $m_{(1,1)}$ for which the minimum is reached verifies: 
$\frac{10^1}{\widehat m_{(1,1)}}=2^{\frac{1}{9}}\times0.2^{\frac{5}{36}}$, \textit{i.e.} $\widehat m_{(1,1)}\approx12$.

Both values of minima (approximately $0.232$ according to Lemma \ref{lem1i}) and ranks for which these minima are reached are correctly approximated by 
our results as illustrated in Table \ref{tab3}, for $d=1$. 

Let us similarly gather in the below table the first values of those ranks and values of local minima, for $d\in\llbracket2,9\rrbracket$.

\begin{table}[ht]
\centering
\begin{tabular}{|c||c|c|c|c|}
\hline
\rowcolor{gray!40}$d$&$\widehat m_{(d,4)}$&$m_{(d,4)}$&$m_d$&$P_{(d,m_{(d,4)})}$\\\hline
$2$& $21643$ & $21642$ & $0.127$ & $0.127$ \\\hline
$3$& $31669$ & $31668$ & $0.088$ & $0.088$ \\\hline
$4$& $41683$ & $41681$ & $0.067$ & $0.067$ \\\hline
$5$& $51692$ & $51690$ & $0.054$ & $0.054$ \\\hline
$6$& $61698$ & $61696$ & $0.046$ & $0.046$ \\\hline
$7$& $71703$ & $71701$ & $0.039$ & $0.039$ \\\hline
$8$& $81706$ & $81704$ & $0.034$ & $0.034$ \\\hline
$9$& $91709$ & $91707$ & $0.031$ & $0.031$ \\\hline
\end{tabular}
\caption{Values of above defined ranks and associated values of local minima of $P_{(d,n)}$, for $n\in\llbracket d\times10^{4},(d+1)
\times10^{4}-1\rrbracket$. We round off values of $P_{(d,n)}$ and $m_d$ to three significant digits and values of 
$\widehat m_{(d,4)}$ to unity.}
\label{tab7}
\end{table}

Indeed, for $n\in\llbracket20000,29999\rrbracket$, our approximation of the rank $m_{(2,4)}$ for which the minimum is reached verifies: 
$\frac{10^4}{\widehat m_{(2,4)}}=\frac{10^{-\frac{10}{9(9\times2-1)}}\big(1-\frac{1}{2+1}\big)^{-\frac{2+1}{9\times2-1}}}{2}
=\frac{10^{-\frac{10}{153}}(\frac{2}{3})^{-\frac{3}{17}}}{2}$, \textit{i.e.} $\widehat m_{(2,4)}\approx21643$.

Both values of minima and ranks for which these minima are reached are correctly approximated by our results as illustrated in Table \ref{tab7}. 

The second function we need to study is defined below:

\begin{lemma}\label{lem2i}The maximum $M_d$ of the function $g_d$ from $[\frac{1}{10d},\frac{1}{d+1}]$ that maps $x$ onto 
$(d+1)\beta_d x-\frac{10}{9}x\big(\ln(x)+\ln(d+1)\big)$ is reached when $x=\frac{10^{\frac{1}{9}}\big(1-\frac{1}{d+1}\big)^d}{d+1}$. Its value is 
$M_d=\frac{10^{\frac{10}{9}}\big(1-\frac{1}{d+1}\big)^d}{9(d+1)}$.
\end{lemma}

\begin{proof}
$\forall x\in[\frac{1}{10d},\frac{1}{d+1}]$, $g_d'(x)=(d+1)\beta_d-\frac{10}{9}(1+\ln(x)+\ln(d+1))$. By solving the equation 
$(d+1)\beta_d-\frac{10}{9}(1+\ln(x)+\ln(d+1))=0$, we obtain $x=\frac{10^{\frac{1}{9}}\big(1-\frac{1}{d+1}\big)^d}{d+1}$.\\
Finally $g_d(\frac{10^{\frac{1}{9}}\big(1-\frac{1}{d+1}\big)^d}{d+1})=\frac{10^{\frac{10}{9}}\big(1-\frac{1}{d+1}\big)^d}{9(d+1)}$.
\end{proof}

\begin{remark}
We note that: 
\begin{align*}
g_d(\frac{1}{d+1})&=\beta_d\qquad\text{and}\\
g_d(\frac{1}{10d})&=(d+1)\beta_d\times\frac{1}{10d}-\frac{10}{9}\times\frac{1}{10d}\big(\ln(\frac{1}{10d})+\ln(d+1)\big)\\
&=\frac{9+\ln10+9d\ln(1-\frac{1}{d+1})}{81d}-\frac{1}{9d}\big(\ln(\frac{1}{10d})+\ln(d+1)\big)\\
&=\frac{9+10\ln10+9(d+1)\ln(1-\frac{1}{d+1})}{81d}=\alpha_d\,.
\end{align*}
\end{remark}

Using the approximation of Proposition \ref{pro3i}, the definition of $(\widehat P_{(d,n)})_{n\in\mathbb N^*}$ (for 
$n\in\llbracket (d+1)\times10^{k_i},d\times10^{k_d+1}-1\rrbracket$) and the properties of $g_d$ (Lemma \ref{lem2i}), we can approximate the ranks of 
local maxima. Let $i$ be a strictly positive integer. Let us denote by $M_{(d,i)}$ the rank of the local maximum of $P_{(d,n)}$ in $\llbracket 
(d+1)\times10^i,(d+1)\times10^{i+1}-1\rrbracket$ and by $\widehat M_{(d,i)}$ the estimate of this rank. Let us gather in the following table the first 
values of those ranks and values of local maxima, for $d=1$.

\begin{table}[ht]
\centering
\begin{tabular}{|c||c|c|c|}
\hline
\rowcolor{gray!40}$i$&$\widehat M_{(1,i)}$&$M_{(1,i)}$&$P_{(1,M_{(1,i)})}$\\\hline
$1$& $31$ & $27$ & $0.402$   \\\hline
$2$& $310$ & $304$ & $0.364$   \\\hline
$3$& $3097$ & $3090$ & $0.359$   \\\hline
$4$& $30971$ & $30963$ & $0.359$   \\\hline
\end{tabular}
\caption{First four values of above defined ranks and associated values of local maxima of $P_{(1,n)}$. We round off values of $P_{(1,n)}$ to three 
significant digits and values of $\widehat M_{(1,i)}$ to unity. Note that $M_1\approx0.359$.}
\label{tab4}
\end{table}

Indeed, for $n\in\llbracket20,99\rrbracket$, our approximation of the rank $M_{(1,1)}$ for which the maximum is reached verifies: 
$\frac{10^1}{n_M}=\frac{5^{\frac{1}{9}}\times0.5^{\frac{8}{9}}}{2}$, \textit{i.e.} $\widehat M_{(1,1)}\approx31$.

Both values of maxima (approximately $0.359$ according to Lemma \ref{lem2i}) and ranks for which these maxima are reached are correctly approximated by 
our results as illustrated in Table \ref{tab4}, for $d=1$.

Let us gather in the following table the first values of those ranks and values of local maxima, for $d\in\llbracket2,9\rrbracket$.

\begin{table}[ht]
\centering
\begin{tabular}{|c||c|c|c|c|}
\hline
\rowcolor{gray!40}$d$&$\widehat M_{(d,4)}$&$M_{(d,4)}$&$M_d$&$P_{(d,M_{(d,4)})}$\\\hline
$2$& $52263$ & $52258$ & $0.213$ & $0.213$ \\\hline
$3$& $73412$ & $73409$ & $0.151$ & $0.151$ \\\hline
$4$& $94515$ & $94515$ & $0.118$ & $0.118$ \\\hline
$5$& $115597$ & $115600$ & $0.096$ & $0.096$ \\\hline
$6$& $136668$ & $136673$ & $0.081$ & $0.081$ \\\hline
$7$& $157733$ & $157741$ & $0.070$ & $0.070$ \\\hline
$8$& $178793$ & $178803$ & $0.062$ & $0.062$ \\\hline
$9$& $199851$ & $199863$ & $0.056$ & $0.056$ \\\hline
\end{tabular}
\caption{Values of above defined ranks and associated values of local maxima of $P_{(d,n)}$, for $n\in\llbracket (d+1)\times10^{4},d
\times10^{5}-1\rrbracket$. We round off values of $P_{(d,n)}$ and $M_d$ to three significant digits and values of $\widehat M_{(d,4)}$ to unity.}
\label{tab8}
\end{table}

Indeed, for $n\in\llbracket30000,199999\rrbracket$, our approximation of the rank $M_{(2,4)}$ for which the maximum is reached verifies: 
$\frac{10^4}{\widehat M_{(2,4)}}=\frac{10^{\frac{1}{9}}\big(1-\frac{1}{2+1}\big)^2}{2+1}=\frac{4\times10^{\frac{1}{9}}}{27}$, \textit{i.e.} 
$\widehat M_{(2,4)}\approx52263$.

Both values of maxima and ranks for which these maxima are reached are correctly approximated by our results as illustrated in Table \ref{tab8}.

When considering the random experiment defined in the beginning of the article, we have thus determined the values of the 
proportions of selected numbers whose leading digit is $d$ and its bounds: these values seem to fluctuate between $m_d$ and 
$M_d$. 

\section{Central values}

From previous Figures, we notice that there exist fluctuations in the graph of $(P_{(d,n)})_{n\in\mathbb N^*}$. We can 
calculate over each "pseudo-cycle", \textit{i.e.} for all $n\in\llbracket d\times10^{i},d\times10^{i+1}-1\rrbracket$ where $i\in\mathbb N$, the mean 
value $C_{(d,i)}$ of $P_{(d,n)}$. For example, we obtain:
\begin{examples}
\begin{align*}
C_{(2,0)}&=\frac{1}{18}\sum_{i=2}^{19}P_{(2,i)}\approx0.197\\
C_{(5,1)}&=\frac{1}{450}\sum_{i=50}^{499}P_{(5,i)}\approx0.074\\
C_{(9,2)}&=\frac{1}{8100}\sum_{i=900}^{8999}P_{(9,i)}\approx0.043\,.
\end{align*}
\end{examples}

We will now consider the sequence $(C_{(d,n)})_{n\in\mathbb N}$ and will demonstrate that it converges. Before we go any further, let us prove the 
following lemma:
\begin{lemma}\label{lem3}
For all $(p,q)\in\mathbb N^2$, such that $4<p<q$, we have:
\begin{align*}
\frac{\ln\big(q(p-1)\big)\ln(\frac{q}{p-1})}{2}\le\sum_{n=p}^qn\ln n\le\frac{\ln\big((q+1)p\big)\ln(\frac{q+1}{p})}{2}\,.
\end{align*}
\end{lemma}

\begin{proof}
The function from $[3;+\infty[$ to $\mathbb R$ that maps $x$ onto $x\ln x$ is increasing on $[3;+\infty[$. Thus:
\begin{align*}
\displaystyle \int_{p-1}^{q} x\ln x \, \mathrm{d}x&\le\sum_{n=p}^qn\ln n\le\displaystyle \int_{p}^{q+1} x\ln x \, \mathrm{d}x\\
\left [  \frac{(\ln x)^2}{2}\right]^q_{p-1}&\le\sum_{n=p}^qn\ln n\le\left [  \frac{(\ln x)^2}{2}\right]_{p}^{q+1}\,.
\end{align*}
The result follows.
\end{proof}

The below proposition can thereupon be stated:
\begin{proposition}\label{proci}
\begingroup\scriptsize
\begin{align*}
C_{(d,n)}\underset{+\infty}{\sim}\frac{(18d(\alpha_d-1)-(9d-1)\ln(\frac{d+1}{d}))\ln(\frac{d+1}{d})+18+2(9(d+1)\beta_d+5\ln(\frac{10d}{d+1}))
\ln(\frac{10d}{d+1})}{162d}\,.
\end{align*}
\endgroup
\end{proposition}

\begin{proof}Let $\mathcal E$ be a real number such that $0<\mathcal E<1$. 

For all $n\in\mathbb N^*$, we have:
\begingroup\footnotesize
\begin{align*}
C_{(d,n)}&=\frac{1}{9d\times10^n}\sum_{i=d\times10^n}^{d\times10^{n+1}-1}P_{(d,i)}\\
&=\frac{1}{9d\times10^n}\Big(\sum_{i=d\times10^n}^{\lfloor (d+\mathcal E)\times10^n\rfloor-1}P_{(d,i)}+\sum_{i=\lfloor (d+\mathcal E)\times10^n\rfloor}
^{(d+1)\times10^n-1}P_{(d,i)}\\
&\enspace+\sum_{i=(d+1)\times10^n}^{\lfloor((d+1)+\mathcal E)\times10^n\rfloor-1}P_{(d,i)}+
\sum_{i=\lfloor((d+1)+\mathcal E)\times10^n\rfloor}^{d\times10^{n+1}-1}P_{(d,i)}\Big)\,.
\end{align*}
\endgroup

For all $n\in\mathbb N^*$, let us consider $\widehat C_{(d,n)}=\frac{1}{9d\times10^n}\sum_{i=d\times10^n}^{d\times10^{n+1}-1}\widehat P_{(d,i)}$. 
We know that:
\begin{align*}
\forall i\in\mathbb N^*,\quad |P_{(d,i)}-\widehat P_{(d,i)}|\le|P_{(d,i)}|+|\widehat P_{(d,i)}|\le2\,.
\end{align*}
There also exists an integer $t$ such that for all $i\ge t$ and $i\in I_{\mathcal{E}}$, $|P_{(d,i)}-\widehat P_{(d,i)}|\le\mathcal E$ (see 
Proposition \ref{pro3i}).

Thus, for all $n\in\mathbb N^*$ such that $10^n\ge t$, we have:
\begingroup\scriptsize
\begin{align*}
|C_{(d,n)}-\widehat C_{(d,n)}|&\le\frac{1}{9d\times10^n}\sum_{i=d\times10^n}^{d\times10^{n+1}-1}|P_{(d,i)}-\widehat P_{(d,i)}|\\
&\le\frac{1}{9d\times10^n}\Big(\sum_{i=d\times10^n}^{\lfloor (d+\mathcal E)\times10^n\rfloor-1}|P_{(d,i)}-\widehat P_{(d,i)}|+
\sum_{i=\lfloor (d+\mathcal E)\times10^n\rfloor}^{(d+1)\times10^n-1}|P_{(d,i)}-\widehat P_{(d,i)}|\\
&\enspace+\sum_{i=(d+1)\times10^n}^{\lfloor((d+1)+\mathcal E)\times10^n\rfloor-1}|P_{(d,i)}-\widehat P_{(d,i)}|
+\sum_{i=\lfloor((d+1)+\mathcal E)\times10^n\rfloor}^{d\times10^{n+1}-1}|P_{(d,i)}-\widehat P_{(d,i)}|\Big)\\
&\le\frac{1}{9d\times10^n}\Big(2(\lfloor (d+\mathcal E)\times10^n\rfloor-d\times10^n)+((d+1)\times10^n-\lfloor (d+\mathcal E)\times10^n\rfloor)
\mathcal E\\
&\enspace+2(\lfloor((d+1)+\mathcal E)\times10^n\rfloor-(d+1)\times10^n)\\
&\enspace+(d\times10^{n+1}-\lfloor((d+1)+\mathcal E)\times10^n\rfloor)\mathcal E\Big)\\
&\le\frac{1}{9d\times10^n}\Big(2\mathcal E\times 10^n +10^n\mathcal E +2\mathcal E\times10^n+8\times10^n\mathcal E\Big)\le\frac{(5+(9d-1))\mathcal E}{9}\,.
\end{align*}
\endgroup
Consequently $C_{(d,n)}\underset{+\infty}{\sim}\widehat C_{(d,n)}$. We will henceforth study $(\widehat C_{(d,n)})_{n\in\mathbb N}$. We have for all 
$n\in\mathbb N^*$:
\begin{align}\label{equ1i}
\widehat C_{(d,n)}&=\frac{1}{9d}\Big(\sum_{i=d\times10^n}^{(d+1)\times10^n-1}\frac{\widehat P_{(d,i)}}{10^n}+\sum_{i=(d+1)\times10^n}^{d\times10^{n+1}-1}
\frac{\widehat P_{(d,i)}}{10^n}\Big)\,.
\end{align}
We consider the first term of this sum:
\begingroup\small
\begin{align*}
\frac{1}{9d}\sum_{i=d\times10^n}^{(d+1)\times10^n-1}\frac{\widehat P_{(d,i)}}{10^n}&=\frac{1}{9d}\sum_{i=d\times10^n}^{(d+1)\times10^n-1}
\Big(\frac{d(\alpha_d-1)}{i}+\frac{1}{10^n}+\frac{9d-1}{9d}(\ln(\frac{10^n}{i})+\ln d)\Big)\\
&=\frac{9d(\alpha_d-1)+(9d-1)n\ln(10)+(9d-1)\ln d}{81d}\sum_{i=d\times10^n}^{(d+1)\times10^n-1}\frac{1}{i}\\
&\enspace-\frac{9d-1}{81d}\sum_{i=d\times10^n}^{(d+1)\times10^n-1}\frac{\ln i}{i}+\frac{1}{9d}\,.
\end{align*}
\endgroup
The proofs of Lemma \ref{lem0i} and Proposition \ref{sub1i} allow us to state:
\begin{align*}
\ln(\frac{(d+1)\times10^n}{d\times10^n})&\le\sum_{i=d\times10^n}^{(d+1)\times10^n-1}\frac{1}{i}\le\ln(\frac{(d+1)\times10^n-1}{d\times10^n-1})\\
\ln(\frac{d+1}{d}) &\le\sum_{i=d\times10^n}^{(d+1)\times10^n-1}\frac{1}{i}\le\ln(\frac{d+1}{d})+\ln(1+\frac{\frac{1}{d(d+1)}}{10^n-\frac{1}{d}})\,,
\end{align*}
\textit{i.e.} $\sum_{i=d\times10^n}^{(d+1)\times10^n-1}\frac{1}{i}=\ln(\frac{d+1}{d})+\underset{+\infty}{O}(\frac{1}{10^n})$. Then thanks to Lemma 
\ref{lem3}:
\begingroup\footnotesize
\begin{align*}
\frac{\ln\big(((d+1)\times10^n-1)(d\times10^n-1)\big)\ln(\frac{(d+1)\times10^n-1}{d\times10^n-1})}{2}&\le\sum_{i=d\times10^n}^{(d+1)\times10^n-1}
\frac{\ln i}{i}\\
\frac{\ln(d(d+1)\times10^{2n}-(2d+1)\times10^n+1)(\ln(\frac{d+1}{d})+\ln(1+\frac{\frac{1}{d(d+1)}}{10^n-\frac{1}{d}}))}{2}&\le
\sum_{i=d\times10^n}^{(d+1)\times10^n-1}\frac{\ln i}{i}\,.
\end{align*}
\endgroup
We have:
\begingroup\footnotesize
\begin{align*}
\ln(d(d+1)\times10^{2n}-(2d+1)\times10^n+1)&=\ln(d(d+1)\times10^{2n})+\ln(1-\frac{(2d+1)-\frac{1}{10^n}}{d(d+1)\times10^n})\\
&=\ln(d(d+1))+2n\ln(10)+\underset{+\infty}{O}(\frac{2d+1}{d(d+1)\times10^n})\,.
\end{align*}
\endgroup
We obtain:
\begingroup\footnotesize
\begin{align*}
\frac{\big(\ln(d(d+1))+2n\ln(10)+\underset{+\infty}{O}(\frac{1}{10^n})\big)\big(\ln(\frac{d+1}{d})+\underset{+\infty}{O}(\frac{1}{10^n})\big)}{2}&\le
\sum_{i=d\times10^n}^{(d+1)\times10^n-1}\frac{\ln i}{i}\\
\frac{(\ln(d(d+1))+2n\ln(10))\ln(\frac{d+1}{d})+\underset{+\infty}{o}(1)}{2}&\le\sum_{i=d\times10^n}^{(d+1)\times10^n-1}\frac{\ln i}{i}\,.
\end{align*}
\endgroup
Thanks to Lemma \ref{lem3} we also have:
\begingroup\footnotesize
\begin{align*}
\sum_{i=d\times10^n}^{(d+1)\times10^n-1}\frac{\ln i}{i}&\le\frac{\ln\big(d(d+1)\times10^{2n}\big)\ln(\frac{(d+1)\times10^n}{d\times10^n})}{2}\\
\sum_{i=d\times10^n}^{(d+1)\times10^n-1}\frac{\ln i}{i}&\le\frac{(\ln(d(d+1))+2n\ln(10))\ln(\frac{d+1}{d})}{2}\,.
\end{align*}
\endgroup
Finally the first term of equality (\ref{equ1i}) verifies: 
\begin{align*}
\frac{1}{9d}\sum_{i=d\times10^n}^{(d+1)\times10^n-1}\frac{\widehat P_{(d,i)}}{10^n}&\underset{+\infty}{\sim}\frac{(9d(\alpha_d-1)+(9d-1)n\ln10+(9d-1)
\ln d)\ln(\frac{d+1}{d})}{81d}\\
&\enspace-\frac{(9d-1)(\ln(d(d+1))+2n\ln(10))\ln(\frac{d+1}{d})}{2\times81d}+\frac{1}{9d}\\
&\underset{+\infty}{\sim}\frac{(18d(\alpha_d-1)-(9d-1)\ln(\frac{d+1}{d}))\ln(\frac{d+1}{d})+18}{162d}\,.
\end{align*}

We consider henceforth the second term of the equality (\ref{equ1i}):
\begingroup\small
\begin{align*}
\frac{1}{9d}\sum_{i=(d+1)\times10^n}^{d\times10^{n+1}-1}\frac{\widehat P_{(d,i)}}{10^n}&=\frac{1}{9d}\sum_{i=(d+1)\times10^n}^{d\times10^{n+1}-1}\Big(
(d+1)\beta_d\frac{1}{i}-\frac{10}{9d}\big(\ln(d+1)+\ln(\frac{10^n}{i})\big)\Big)\\
&=\frac{9(d+1)\beta_d-10\ln(d+1)-10\ln(10)n}{81d}\sum_{i=(d+1)\times10^n}^{d\times10^{n+1}-1}\frac{1}{i}\\
&\enspace+\frac{10}{81d}\sum_{i=(d+1)\times10^n}^{d\times10^{n+1}-1}\frac{\ln i}{i}\,.
\end{align*}
\endgroup
The proofs of Lemma \ref{lem0'i} and Proposition \ref{sub2i} allow us to state:
\begin{align*}
\ln(\frac{d\times10^{n+1}}{(d+1)\times10^n})&\le\sum_{i=(d+1)\times10^n}^{d\times10^{n+1}-1}\frac{1}{i}\le\ln(\frac{d\times10^{n+1}-1}
{(d+1)\times10^n-1})\\
\ln(\frac{10d}{d+1})&\le\sum_{i=(d+1)\times10^n}^{d\times10^{n+1}-1}\frac{1}{i}\le\ln(\frac{10d}{d+1})+
\ln(1+\frac{\frac{1}{d+1}-\frac{1}{10i}}{10^n-\frac{1}{d+1}})\,,
\end{align*}
\textit{i.e.} $\sum_{i=(d+1)\times10^n}^{d\times10^{n+1}-1}\frac{1}{i}=\ln(\frac{10d}{d+1})+\underset{+\infty}{O}(\frac{1}{10^n})$. Then thanks to 
Lemma \ref{lem3}:
\begingroup\scriptsize
\begin{align*}
\frac{\ln\big((d\times10^{n+1}-1)((d+1)\times10^n-1)\big)\ln(\frac{d\times10^{n+1}-1}{(d+1)\times10^n-1})}{2}&\le\sum_{i=(d+1)\times10^n}^
{d\times10^{n+1}-1}\frac{\ln i}{i}\\
\frac{\ln(10d(d+1)\times10^{2n}-(11d+1)\times10^n+1)(\ln(\frac{10d}{d+1})+\ln(1+\frac{\frac{1}{d+1}-\frac{1}{10d}}{10^n-\frac{1}{d+1}}))}{2}&\le
\sum_{i=(d+1)\times10^n}^{d\times10^{n+1}-1}\frac{\ln i}{i}\,.
\end{align*}
\endgroup
We have:
\begingroup\scriptsize
\begin{align*}
\ln(10d(d+1)10^{2n}-(11d+1)10^n+1)&=\ln(10d(d+1)10^{2n})+\ln(1-\frac{11d+1-\frac{1}{10^n}}{10d(d+1)10^{n+1}})\\
&=\ln(10d(d+1))+2n\ln(10)+\underset{+\infty}{O}(\frac{11d+1}{10d(d+1)\times10^{n+1}})\,.
\end{align*}
\endgroup
We obtain:
\begingroup\scriptsize
\begin{align*}
\frac{\big(\ln(10d(d+1))+2n\ln(10)+\underset{+\infty}{O}(\frac{1}{10^n})\big)\big(\ln(\frac{10d}{d+1})+\underset{+\infty}{O}(\frac{1}{10^n})\big)}{2}&
\le\sum_{i=(d+1)\times10^n}^{d\times10^{n+1}-1}\frac{\ln i}{i}\\
\frac{(\ln(10d(d+1))+2n\ln10)\ln(\frac{10d}{d+1})+\underset{+\infty}{o}(1)}{2}&\le
\sum_{i=(d+1)\times10^n}^{d\times10^{n+1}-1}\frac{\ln i}{i}\,.
\end{align*}
\endgroup
Thanks to Lemma \ref{lem3} we also have:
\begingroup\scriptsize
\begin{align*}
\sum_{i=(d+1)\times10^n}^{d\times10^{n+1}-1}\frac{\ln i}{i}&\le\frac{\ln(d(d+1)\times10^{2n+1})\ln(\frac{10d}{d+1})}{2}\\
\sum_{i=(d+1)\times10^n}^{d\times10^{n+1}-1}\frac{\ln i}{i}&\le\frac{(\ln(10d(d+1))+2n\ln(10))\ln(\frac{10d}{d+1})}{2}\,.
\end{align*}
\endgroup
Finally the second term of equality (\ref{equ1i}) verifies:
\begingroup\small
\begin{align*}
\frac{1}{9d}\sum_{i=(d+1)\times10^n}^{d\times10^{n+1}-1}\frac{\widehat P_{(d,i)}}{10^n}&\underset{+\infty}{\sim}\frac{(9(d+1)\beta_d-10\ln(d+1)-
10n\ln(10))\ln(\frac{10d}{d+1})}{81d}\\
&\enspace+\frac{5(\ln(10d(d+1))+2n\ln(10))\ln(\frac{10d}{d+1})}{81d}\\
&\underset{+\infty}{\sim}\frac{(9(d+1)\beta_d+5\ln(\frac{10d}{d+1}))\ln(\frac{10d}{d+1})}{81d}\,.
\end{align*}
\endgroup

Hence:
\begingroup\scriptsize
\begin{align*}
\widehat C_{(1,n)}&\underset{+\infty}{\sim}\frac{(18d(\alpha_d-1)-(9d-1)\ln(\frac{d+1}{d}))\ln(\frac{d+1}{d})+18}{162d}+\frac{(9(d+1)\beta_d
+5\ln(\frac{10d}{d+1}))\ln(\frac{10d}{d+1})}{81d}\\
&\underset{+\infty}{\sim}\frac{(18d(\alpha_d-1)-(9d-1)\ln(\frac{d+1}{d}))\ln(\frac{d+1}{d})+18+2(9(d+1)\beta_d+5\ln(\frac{10d}{d+1}))
\ln(\frac{10d}{d+1})}{162d}\,.
\end{align*}
\endgroup
The result follows.
\end{proof}

Let us denote by $C_d$ the limit of $(C_{(d,n)})_{n\in\mathbb N}$.

Note that our first choice of "pseudo-cycle" give more weight at the values of proportion situated at the end of intervals 
$\llbracket d\times10^{i},d\times10^{i+1}-1\rrbracket$ where $i\in\mathbb N$. We can also have defined the sequence 
$(\tilde C_{(d,n)})_{n\in\mathbb N}$ as follows:
\begin{align*}
\tilde C_{(d,n)}=\frac{1}{9(d+1)\times10^n}\sum_{i=(d+1)\times10^n}^{(d+1)\times10^{n+1}-1}P_{(1,i)}\,.
\end{align*}

If so, we would state that the limit $\tilde C_d$ of the sequence $(\tilde C_{(d,n)})_{n\in\mathbb N}$ is:
\begingroup\footnotesize
\begin{align*}
\frac{(90d(\alpha_d-1)-5(9d-1)\ln(\frac{d+1}{d}))\ln(\frac{d+1}{d})+90+
(9(d+1)\beta_d+5\ln(\frac{10d}{d+1}))\ln(\frac{10d}{d+1})}{81(d+1)}\,.
\end{align*}
\endgroup

\begin{proof}$\tilde C_{(d,n)}\underset{+\infty}{\sim}\widehat C_{(d,n)}$ where, for all $n\in\mathbb N^*$:
\begin{align}\label{equit}
\widehat C_{(d,n)}&=\frac{1}{9(d+1)}\Big(\sum_{i=(d+1)\times10^n}^{d\times10^{n+1}-1}\frac{\widehat P_{(d,i)}}{10^n}+\sum_{i=d\times10^{n+1}}^
{(d+1)\times10^{n+1}-1}\frac{\widehat P_{(d,i)}}{10^n}\Big)\,.
\end{align}
The first term of this sum is similar to the second term of the associated sum in the proof of Proposition \ref{proci}. The second one verifies:
\begingroup\scriptsize
\begin{align*}
\frac{1}{9(d+1)}\sum_{i=d\times10^{n+1}}^{(d+1)\times10^{n+1}-1}\frac{\widehat P_{(d,i)}}{10^n}&=
\frac{90d(\alpha_d-1)+10(9d-1)((n+1)\ln(10)+\ln d)}{81(d+1)}\sum_{i=d\times10^{n+1}}^{(d+1)\times10^{n+1}-1}\frac{1}{i}\\
&\enspace-\frac{10(9d-1)}{81(d+1)}\sum_{i=d\times10^{n+1}}^{(d+1)\times10^{n+1}-1}\frac{\ln i}{i}+\frac{10}{9(d+1)}\,.
\end{align*}
\endgroup

Finally the second term of equality (\ref{equit}) is: 
\begingroup\scriptsize
\begin{align*}
\frac{1}{9(d+1)}\sum_{i=d\times10^{n+1}}^{(d+1)\times10^{n+1}-1}\frac{\widehat P_{(d,i)}}{10^n}&\underset{+\infty}{\sim}
\frac{(90d(\alpha_d-1)+10(9d-1)((n+1)\ln10+\ln d))\ln(\frac{d+1}{d})}{81(d+1)}\\
&\enspace-\frac{5(9d-1)(\ln(d(d+1))+2(n+1)\ln(10))\ln(\frac{d+1}{d})}{81(d+1)}+\frac{10}{9(d+1)}\\
&\underset{+\infty}{\sim}\frac{(90d(\alpha_d-1)-5(9d-1)\ln(\frac{d+1}{d}))\ln(\frac{d+1}{d})+90}{81(d+1)}\,.
\end{align*}
\endgroup

Hence:
\begingroup\scriptsize
\begin{align*}
\widehat C_{(d,n)}&\underset{+\infty}{\sim}\frac{(9(d+1)\beta_d+5\ln(\frac{10d}{d+1}))\ln(\frac{10d}{d+1})}{81(d+1)}+
\frac{(90d(\alpha_d-1)-5(9d-1)\ln(\frac{d+1}{d}))\ln(\frac{d+1}{d})+90}{81(d+1)}\\
&\underset{+\infty}{\sim}\frac{(90d(\alpha_d-1)-5(9d-1)\ln(\frac{d+1}{d}))\ln(\frac{d+1}{d})+90+(9(d+1)\beta_d+
5\ln(\frac{10d}{d+1}))\ln(\frac{10d}{d+1})}{81(d+1)}\,.
\end{align*}
\endgroup
The result follows.
\end{proof}

Once more, means values over both "pseudo-cycles" are very close to the theoric value highlighted by Benford: $\log(1+\frac{1}{d})$ (\cite{ben}). Table 
\ref{tab9} below gathers the whole values:

\begin{table}[ht]
\centering
\begin{tabular}{|c||c|c|c|}
\hline
\rowcolor{gray!40}$d$&$\tilde C_d$&$\log(1+\frac{1}{d})$&$C_d$\\\hline
$1$& $0.281$ & $0.301$ & $0.301$  \\\hline
$2$& $0.160$ & $0.176$ & $0.191$  \\\hline
$3$& $0.113$ & $0.125$ & $0.139$  \\\hline
$4$& $0.088$ & $0.097$ & $0.109$  \\\hline
$5$& $0.072$ & $0.079$ & $0.090$  \\\hline
$6$& $0.061$ & $0.067$ & $0.077$ \\\hline
$7$& $0.053$ & $0.058$ & $0.067$  \\\hline
$8$& $0.047$ & $0.051$ & $0.059$  \\\hline
$9$& $0.042$ & $0.046$ & $0.053$  \\\hline
\end{tabular}
\caption{Values of $C_d$, $\tilde C_d$ and Benford's Law probabilities, these values being rounded to the nearest thousandth.}
\label{tab9}
\end{table}

Indeed, according to Hill (\cite{hip}), it is absolutely normal. In a way, it can be considered as an equivalent to the central limit theorem 
(\cite{del}).

\section*{Conclusion}

To conclude, through our model, we have seen that the proportion of $d$ as leading digit, $d\in\llbracket1,9\rrbracket$, in certain naturally occurring 
collections of data is more likely to follow a law whose probability distribution is $(d,P_{(d,n)})_{d\in\llbracket1,9\rrbracket}$, where $n$ is 
the smaller integer upper bound of the physical, biological or economical quantities considered, rather than Benford's Law. These probability 
distributions fluctuate around Benford's value as can be seen in the literature (see \cite{knu}, \cite{BK}, \cite{NW} or \cite{FGP} for example) in 
accordance with our model. Knowing beforehand the value of the upper bound $n$ can be a way to find a better adjusted law than Benford's one. 

The results of the article would have been the same in terms of fluctuations of the proportion of $d\in\llbracket1,9\rrbracket$ as leading digit, of 
limits of subsequences, or of results on central values, if our discrete uniform distributions uniformly randomly selected were lower bounded by a 
positive integer different from $0$: first terms in proportion formulas become rapidly negligible. Through our model we understand that the predominance 
of $1$ as first digit (followed by those of $2$ and so on) is all but surprising in experimental data: it is only due to the fact that, in the 
lexicographical order, $1$ appears before $2$, $2$ appears before $3$, \textit{etc.}

However the limits of our model rest on the assumption that the random variables used to obtain our data are not the same and follow discrete uniform 
distributions that are uniformly randomly selected. In certain naturally occurring collections of data it cannot conceivably be justified. Studying the 
cases where the random variables follow other distributions (and not necessarily randomly selected) sketch some avenues for future research on the 
subject.

\bibliography{bib}

\begin{thebibliography}{10}

\bibitem{bee}
T.~W. Beer.
\newblock Terminal digit preference: beware of benford's law.
\newblock {\em Journal of Clinical Pathology}, 62(2):192, 2009.

\bibitem{ben}
F.~Benford.
\newblock The law of anomalous numbers.
\newblock {\em Proceedings of the American Philosophical Society}, 78:127--131,
  1938.

\bibitem{ber}
Bunimovich L. Hill~T. Berger, A.
\newblock One-dimensional dynamical systems and benford’s law.
\newblock {\em Transactions of the American Mathematical Society},
  357(1):197--219, 2004.

\bibitem{boy}
J.~Boyle.
\newblock An application of fourier series to the most significant digit
  problem.
\newblock {\em America Mathematical Monthly}, 101:879--886, 1994.

\bibitem{BK}
J.~Burke and E.~Kincanon.
\newblock Benford’s law and physical constants: the distribution of initial
  digits.
\newblock {\em American {J}ournal of Physics}, 59:952, 1991.

\bibitem{die}
A.~Diekmann.
\newblock Not the first digit! using benford's law to detect fraudulent
  scientific data.
\newblock {\em {J}ournal of Applied Statistics}, 34(3):321--329, 2007.

\bibitem{FGP}
J.~L. Friar, T.~Goldman, and J.~Pérez-Mercader.
\newblock Genome sizes and the {B}enford distribution.
\newblock {\em PLOS ONE}, 7(5), 2012.

\bibitem{del}
N.~Gauvrit and J.-P. Delahaye.
\newblock Pourquoi la loi de benford n'est pas mystérieuse.
\newblock {\em Mathématiques et sciences humaines}, 182(2):7--15, 2008.

\bibitem{hil}
T.~Hill.
\newblock Base-invariance implies benford’s law.
\newblock {\em Proceedings of the American Mathematical Society}, 123:887--895,
  1995.

\bibitem{hip}
T.~Hill.
\newblock A statistical derivation of the significant-digit law.
\newblock {\em Statistical Science}, 10(4):354--363, 1995.

\bibitem{jol}
P.~Jolissaint.
\newblock Loi de benford, relations de récurrence et suites
  équi-distribuées.
\newblock {\em Elemente der Mathematik}, 60(1):10--18, 2005.

\bibitem{knu}
D.~Knuth.
\newblock {\em The Art of Computer Programming 2}.
\newblock Addison-Wesley, New-York, 1969.

\bibitem{new}
R.~Newcomb.
\newblock {N}ote on the frequency of use of the different digits in natural
  numbers.
\newblock {\em American Journal of Mathematics}, 4:39--40, 1881.

\bibitem{NW}
M.~Nigrini and W.~Wood.
\newblock Assessing the integrity of tabulated demographic data.
\newblock 1995.
\newblock Preprint.

\bibitem{pin}
R.~S. Pinkham.
\newblock On the distribution of first significant digits.
\newblock {\em The Annals of Mathematical Statistics}, 32(4):1223--1230, 1961.

\bibitem{rai}
R.~A. Raimi.
\newblock The first digit problem.
\newblock {\em American Mathematical Monthly}, 83(7):521--538, 1976.

\bibitem{sar}
P.~B. Sarkar.
\newblock An observation on the significant digits of binomial coefficients and
  factorials.
\newblock {\em Sankhya B}, 35:362--364, 1973.

\bibitem{pyt}
G.~Van~Rossum.
\newblock {\em Python tutorial}, volume {T}echnical {R}eport CS-R9526.
\newblock 1995.
\newblock Centrum voor Wiskunde en Informatica (CWI).

\bibitem{var}
H.~Varian.
\newblock Benford's law (letters to the editor).
\newblock {\em The American Statistician}, 26(3):62--65, 1972.

\bibitem{was}
L.~C. Washington.
\newblock Benford's law for fibonacci and lucas numbers.
\newblock {\em The Fibonacci Quarterly}, 19(2):175--177, 1981.

\end{thebibliography}

\section*{Appendix: Python script}

Using Propositions \ref{proi}, we can determine the terms of $(P_{(d,n)})_{n\in\mathbb N^*}$, for $d\in\llbracket1,9\rrbracket$. To this 
end, we have created a script with the Python programming language (Python Software Foundation, Python Language Reference, version $3.4.$ available at 
\url{http://www.python.org}, see \cite{pyt}). The implemented function \textit{expvalProp} has two parameters: the 
rank $n$ of the wanted term of the sequence and the value $ld$ of the considered leading digit. Here is the used algorithm:
\vspace{0.5\baselineskip}

{\setlength{\parindent}{1.5em}\textit{def expvalProp(n,ld):}}

{\setlength{\parindent}{3em}\textit{if(ld$>$n):}}

{\setlength{\parindent}{4.5em}\textit{return(0)}}

{\setlength{\parindent}{3em}\textit{else:}}

{\setlength{\parindent}{4.5em}\textit{k=0}}

{\setlength{\parindent}{4.5em}\textit{while(ld*10**k$<=$n):}}

{\setlength{\parindent}{6em}\textit{k=k+1}}

{\setlength{\parindent}{4.5em}\textit{u=1;v=ld-1;S=0;T=0;}}

{\setlength{\parindent}{4.5em}\textit{for i in range(0,k-1):}}

{\setlength{\parindent}{6em}\textit{for b in range(ld*10**i,(ld+1)*10**i):}}

{\setlength{\parindent}{7.5em}\textit{T=T+(b-v)/b}}

{\setlength{\parindent}{6em}\textit{for a in range((ld+1)*10**i,ld*10**(i+1)):}}

{\setlength{\parindent}{7.5em}\textit{S=S+u/a}}

{\setlength{\parindent}{6em}\textit{u=u+10**(i+1)}}

{\setlength{\parindent}{6em}\textit{v=v*10+8}}

{\setlength{\parindent}{4.5em}\textit{if (n$<$(ld+1)*10**(k-1)):}}

{\setlength{\parindent}{6em}\textit{for b in range(ld*10**(k-1),n+1):}}

{\setlength{\parindent}{7.5em}\textit{T=T+(b-v)/b}}

{\setlength{\parindent}{4.5em}\textit{else:}}

{\setlength{\parindent}{6em}\textit{for b in range(ld*10**(k-1),(ld+1)*10**(k-1)):}}

{\setlength{\parindent}{7.5em}\textit{T=T+(b-v)/b}}

{\setlength{\parindent}{6em}\textit{for a in range((ld+1)*10**(k-1),n+1):}}

{\setlength{\parindent}{7.5em}\textit{S=S+u/a}}

{\setlength{\parindent}{4.5em}\textit{return((S+T)/n)}}

\end{document}